\documentclass[12pt]{amsart}
\usepackage{amsmath,latexsym,amsfonts,amssymb,amsthm}

\usepackage{geometry}
\usepackage{comment}
\usepackage{hyperref}
\usepackage{mathrsfs}
\usepackage{tikz-cd}
\usepackage[alphabetic]{amsrefs}
\usepackage[all,cmtip]{xy}
\usepackage{bbm}

\usepackage{bm}

\newtheorem{prop}{Proposition}[section]
\newtheorem{thm}[prop]{Theorem}
\newtheorem{cor}[prop]{Corollary}

\newtheorem{lem}[prop]{Lemma}

\theoremstyle{definition}

\newtheorem{defn}[prop]{Definition}

\newtheorem{rem}[prop]{\it Remark}

\newtheorem*{claim*}{Claim}

\newcommand{\bP}{\mathbb{P}}
\newcommand{\bC}{\mathbb{C}}
\newcommand{\bR}{\mathbb{R}}

\newcommand{\bQ}{\mathbb{Q}}

\newcommand{\bN}{\mathbb{N}}

\newcommand{\Rr}{\mathbb{R}}
\newcommand{\Qq}{\mathbb{Q}}

\newcommand{\cO}{\mathcal{O}}

\newcommand{\fa}{\mathfrak{a}}

\newcommand{\Center}{\mathrm{Center}}

\newcommand{\Supp}{\mathrm{Supp}}

\newcommand{\mult}{\mathrm{mult}}
\newcommand{\lct}{\mathrm{lct}}
\newcommand{\Ex}{\mathrm{Ex}}
\newcommand{\Pic}{\mathrm{Pic}}

\newcommand{\vol}{\mathrm{vol}}
\newcommand{\ord}{\mathrm{ord}}
\newcommand{\Val}{\mathrm{Val}}
\newcommand{\Nklt}{\mathrm{Nklt}}

\newcommand{\hvol}{\widehat{\rm vol}}
\newcommand{\hVol}{\widehat{\rm Vol}}

\newcommand{\Coef}{\mathrm{Coeff}}
\newcommand{\Diff}{\mathrm{Diff}}
\newcommand{\mld}{\mathrm{mld}}

\newcommand{\alg}{\mathrm{alg}}

\numberwithin{equation}{section}

\title{Effective termination of general type MMPs in dimension at most five}

\author{Jingjun Han}
\address{Shanghai Center for Mathematical Sciences, Fudan University, Jiangwan Campus, Shanghai, 200438, China}
\email{hanjingjun@fudan.edu.cn}

\author{Jihao Liu}
\address{Department of Mathematics, Peking University, No. 5 Yiheyuan Road, Haidian District, Beijing 100871, China}
\email{liujihao@math.pku.edu.cn}

\author{Ziquan Zhuang}
\address{Department of Mathematics, Johns Hopkins University, Baltimore, MD 21218, USA}
\email{zzhuang@jhu.edu}

\date{}

\begin{document}

\begin{abstract}
We prove the effective termination of general type MMPs in dimension at most $5$, and give explicit bounds (in terms of topological invariants and volume of divisors) on the number of steps in the MMP when the dimension is at most $4$.
\end{abstract}

\maketitle

\pagestyle{myheadings}\markboth{\hfill J. Han, J. Liu, and Z. Zhuang \hfill}{\hfill Effective termination of general type MMP in dimension at most five \hfill}


\emph{Throughout this paper, we work over $\mathbb{C}$, the field of complex numbers.} 

\section{Introduction}


The termination of the minimal model program (MMP) is one of the major open problems in birational geometry. The termination is known in full generality in dimension three by \cites{Sho85,K+-flip-abundance,Kaw-termination,Sho96}. Several special cases are also known in dimension four \cites{KMM87,Fuj05,AHK07,Bir07,Sho09,Bir10,HL22-Zariski,CT23,Mor25}, and in higher dimensions we have the termination of MMP with scaling (a special kind of MMP) in the general type or big boundary case \cite{BCHM}. Beyond these, the termination conjecture remains extremely difficult. In particular, it is an open question whether every fivefold general type MMP terminates. 

Our first result in this paper is the effective termination of general type MMPs, or more generally, of pseudo-effective MMP with a big boundary, in dimension five.

\begin{thm}[see Theorem \ref{thm: eff termination for gt}] \label{main:termination}
Let $(X,\Delta)$ be a projective klt pair of dimension $5$. Assume that at least one of the following holds.
\begin{enumerate}
    \item $K_X+\Delta$ is big.
    \item $K_X+\Delta$ is pseudo-effective and $\Delta$ is big.
\end{enumerate}
Then there exists a positive integer $m$ depending only on $(X,\Delta)$ such that every $(K_X+\Delta)$-MMP terminates after at most $m$ steps.
\end{thm}

We also give some explicit bounds on the number of steps in the MMP in dimensions $3$ and $4$. A sample result is the following. Here $h_4^{\alg}(X)$ is the dimension of the subspace in the singular homology group
$H_4(X,\mathbb R)$ generated by algebraic cycles.

\begin{thm}[see Theorems \ref{thm:4fold explicit boundary big}] \label{main: explicit}
Let $N$ be a positive integer and $(X,\Delta)$ a projective log smooth klt pair of dimension $4$ such that $N\Delta$ has integer coefficients. Assume that 
\begin{enumerate}
    \item $(X,\Delta)$ has at most $N$ strata, each stratum has Picard number at most $N$, and $h_4^{\alg}(X)\le N$, and
    \item $\Delta$ is big, $\vol(\Delta)\ge \frac{1}{N}$, and $(\Delta\cdot H^3)\le N$ for some very ample divisor $H$ on $X$.
\end{enumerate}
Then any $(K_X+\Delta)$-MMP terminates after at most $M^M$ steps, where $M:=(2 N)^9!$.
\end{thm}

By passing to log resolutions, Theorem \ref{main: explicit} gives explicit termination results for klt MMP with a big boundary (see Lemma \ref{lem:lift MMP dlt}). 

We remark that even the termination of such fourfold MMPs was not previously known. We also prove similar explicit termination results for general type fourfold MMPs (Theorem \ref{thm:4fold explicit general type}), terminal fourfold MMPs (Corollary \ref{cor:terminal explicit}) and arbitrary threefold MMPs (Theorem \ref{thm: threefold explcit upper bound}). For some past works on the explicit or effective termination problem for threefolds, see for example \cites{CH11,Mar22,CC24,Kim25}.

\medskip

\noindent\textbf{Sketch of proof.} A key ingredient in the proof of Theorems \ref{main:termination} and \ref{main: explicit} is some (explicit) uniform bounds on the Cartier index of Weil divisors in a general type MMP \cite[Theorem 3.2, Corollaries 3.18 and 3.20]{HQZ25}. By Birkar's inductive approach \cite{Bir07}, if termination holds in dimension $n-1$, then it also holds for pairs of nonnegative Kodaira dimension in dimension $n$. We refine Birkar's argument to show that if termination holds for MMPs with discrete discrepancies in dimension $n-1$, then we also have termination for MMPs with uniform index bound in dimension $n$. By the termination of terminal fourfold flips, it is not hard to see that MMP with discrete discrepancies always terminates in dimension $4$, and this gives the termination of general type MMPs in Theorem \ref{main:termination}. Once we have termination, the effective bound on the number of steps essentially comes from K\H{o}nig's lemma.

The proof of Theorem \ref{main: explicit} divides into two main steps: we first find explicit bounds on the number of steps in terminal MMPs, and then use the index bound from \cite{HQZ25} to reduce to the terminal case. By a variation of difficulty function type arguments for termination (see e.g. \cites{Sho85,KMM87,K+-flip-abundance,Kaw-termination,Sho96,Fuj05,AHK07}), one may write down a sequence of invariants involving the Picard numbers of boundary divisors, the (weighted) difficulties of the pair, and certain algebraic Betti numbers, that is lexicographically decreasing in any terminal MMPs in dimensions $3$ and $4$. To turn this into an explicit bound on the number of steps in the terminal MMP, we show in Section \ref{ss:explicit terminal} that the above lexicographic sequence of invariants decreases in a bounded way, in the sense that each time an individual term drops by one, the terms thereafter increases by at most a uniformly fixed amount. To deal with the general klt case, we further introduce a weighted combination of the above invariants, where the weights ultimately depends on the index bound from \cite{HQZ25}, and combine with the results in the terminal case to show that the weighted combination decreases by (at least) a fixed amount after each step of the MMPs we consider. This gives the explicit bound in Theorem \ref{main: explicit}.

\medskip

\noindent\textbf{Outline of the article.} In Section \ref{sec: preliminary} we summarize some preliminary results that will be used in the later part of this paper. In Section \ref{sec:tof dim 5} we prove Theorem \ref{main:termination}. In Section \ref{sec:explicit tof} we prove Theorem \ref{main: explicit} and its variants. 
 
\medskip

\noindent\textbf{Acknowledgment} 
The authors would like to thank Jungkai A. Chen, Christopher D. Hacon, James M\textsuperscript{c}Kernan, Lu Qi, Vyacheslav Shokurov, Chenyang Xu for helpful discussions and comments. JH is supported by NSFC for Excellent Young Scientists (\#12322102), and the National Key R\&D Program of China (\#2023YFA1010600). JH is a member of LMNS, Fudan University. JL is supported by the National Key R\&D Program of China (\#2024YFA1014400). ZZ is partially supported by the NSF Grants DMS-2240926, DMS-2234736, a Sloan research fellowship, a Packard fellowship, and the Simons Collaboration on Moduli of Varieties.

\section{Preliminary}\label{sec: preliminary}

In this section, we collect some preliminary materials that will be used throughout the article.

\subsection{Notation and conventions}

In this paper, varieties are assumed to be quasi-projective unless otherwise specified. We follow the standard terminology from \cites{KM98,Kol13}.

Given two $\bR$-divisors $\Delta$ and $\Delta'$ on a variety $X$, we write $\Delta'\ge \Delta$ if $\mult_D\Delta'\geq\mult_D\Delta$ for any prime divisor $D$ on $X$. 

A \emph{sub-pair} $(X,\Delta)$ consists of a normal variety $X$ together with an $\bR$-divisor $\Delta$ on $X$ (a priori, we do not require that $K_X+\Delta$ is $\bR$-Cartier). It is called a \emph{pair} if $\Delta\geq 0$. We denote by $\Coef(\Delta)$ the (finite) set of coefficients of $\Delta$ (in particular, we require that $0\in\Coef(\Delta)$). A \emph{log (sub-)pair} is a (sub-)pair $(X,\Delta)$ such that $K_X+\Delta$ is $\mathbb R$-Cartier. 

A sub-pair $(X,\Delta)$ is called \emph{log smooth} if $X$ is smooth and $\Supp(\Delta)$ is a simple normal crossing divisor. A \emph{stratum} (resp. \emph{open stratum}) of a log smooth pair $(X,\Delta=\sum_{i\in I} b_i B_i)$, where $B_i$ are the irreducible components of $\Delta$, is defined to be a connected component of $B_J:=\cap_{j\in J} B_j$ (resp. $B^\circ_J:=\cap_{j\in J} B_j \setminus \cup_{j\not\in J} B_j$) for some $J\subseteq I$ (by convention, $B_\emptyset = X$ and $B^\circ_\emptyset = X\setminus \Supp(\Delta)$).

Let $(X,\Delta)$ be a log sub-pair. A proper birational morphism $\pi\colon (Y,\Delta_Y)\rightarrow (X,\Delta)$ from a sub-pair is called \emph{crepant} if $K_Y+\Delta_Y=\pi^*(K_X+\Delta)$, in which case we also call $(Y,\Delta_Y)$ the \emph{crepant pullback} of $(X,\Delta)$ (via $\pi$). 

An $\bR$-divisor $D$ on a variety $X$ is \emph{big} if $D=A+E$ for some ample $\bR$-divisor $A$ and some $E\ge 0$. Let $X\to Z$ be a morphism. We say $D$ is \emph{big over $Z$} if $D\sim_{\bR,Z} D'$ for some big $\bR$-divisor $D'$ on $X$. 


An $\bR$-divisor $D$ on a projective variety $X$ is \emph{pseudo-effective} if its numerical class is a limit of $\bR$-divisors $D_i\ge 0$ ($i=1,2,\dots$). Given a projective morphism $X\to Z$, we say $D$ is \emph{pseudo-effective over $Z$} if its restriction to the fiber $X_\eta$ is pseudo-effective, where $\eta\in Z$ is the image of the generic point of $X$.

A \emph{fibration} is a proper surjective morphism $f\colon X\to Z$ of normal varieties such that $f_* \cO_X=\cO_Z$. If $X\to Z$ is a fibration and $(X,\Delta)$ is a pair, by abuse of notation we also call $(X,\Delta)\to Z$ a fibration.


Let $X$ be a variety. For any proper birational morphism $\pi\colon Y\to X$ and any prime divisor $E$ on $Y$, we call $E$ a \emph{(prime) divisor over} $X$ and denote by $\Center_X(E):=\pi(E)$ the \emph{center} of $E$ on $X$. We also denote by $c_X(E)$ the generic point of $\Center_X(E)$ and sometimes also call it the center of $E$ if no confusion arises. We also call $E$ as a \emph{(prime) divisor over} $x\in X$ where $x$ is the generic point of $\pi(E)$. We denote by $\Ex(\pi)\subseteq Y$ the exceptional locus of $\pi$. The \emph{discrepancy} of $E$ with respect to a log sub-pair $(X,\Delta)$, denoted by $a(E,X,\Delta)$, is defined by
\[
a(E,X,\Delta):=\mult_E (K_Y-\pi^*(K_X+\Delta)).
\]

The notation of \emph{terminal}, \emph{klt}, \emph{dlt}, and \emph{log canonical} (sub-)pairs are defined as in \cite[Definitions 2.34 and 2.37]{KM98}. 
Let $(X,\Delta)$ be a log canonical pair. A subvariety $W\subseteq X$ is called a \emph{log canonical center} if either $W=X$, or $W$ is the center of some divisor $E$ over $X$ with $a(E,X,\Delta)=-1$.



\begin{defn}
For any log canonical sub-pair $(X,\Delta)$, we define
\[
\mld(X,\Delta):=\inf\{a(E,X,\Delta)+1\mid E\text{ is a prime divisor over } X\} 
\]
as the \emph{total minimal log discrepancy} (\emph{total mld}) of $(X,\Delta)$.
\end{defn}

\subsection{Minimal Model Program (MMP)}

We refer to e.g. \cite[Paragraph 3.31]{KM98} for a description of the MMP. In this subsection, we recall some results on the MMP that will be used in this paper.

\begin{defn} \label{defn:MMP type contraction}
A birational map $\varphi\colon (X,\Delta)\dashrightarrow (X',\Delta')$ of log pairs is called an \emph{MMP type contraction} if $\varphi^{-1}$ has no exceptional divisor,  and there exist proper birational morphisms $f\colon Y\to X$, $g\colon Y\to X'$ such that
\[
f^*(K_X+\Delta)-g^*(K_{X'}+\Delta')\ge 0
\]
and $g=\varphi\circ f$. In particular, we have $\varphi_*\Delta\ge \Delta'$. 
\end{defn}

As the name suggests, if $\varphi\colon (X,\Delta)\dashrightarrow (X',\Delta')$ is obtained by running a $(K_X+\Delta)$-MMP, then $\varphi\colon (X,\Delta)\dashrightarrow (X',\Delta')$ is an MMP type contraction. The following index bound is proved in \cite{HQZ25}.

\begin{thm} \label{thm:index bdd}
Let $(X,\Delta)$ be a klt pair where $\Delta$ is big. Then there exist a positive integer $r$ depending only on $(X,\Delta)$ such that for any MMP type contraction $\varphi\colon (X,\Delta)\dashrightarrow (X',\Delta')$ with $\Delta'=\varphi_* \Delta$, the Cartier index of any $\bQ$-Cartier Weil divisor on $X'$ is at most $r$.
\end{thm}

\begin{proof}
This is a special case of \cite[Theorem 3.2]{HQZ25}.
\end{proof}

We will frequently use the following observation, which is well known to the experts ({\it cf.} \cite[Lemma 3.7.5 and Proof of Theorem 1.2]{BCHM}), to reduce questions about general type MMPs to MMPs with a big boundary.

\begin{lem}\label{lem: gt klt sim to big boundary}
Let $(X,\Delta)$ be a klt pair and $X\rightarrow T$ a morphism. 
\begin{enumerate}
    \item If $K_X+\Delta$ is big over $T$, then there exists a klt pair $(X,\Delta')$, such that $K_X+\Delta'\sim_{\Rr,T} \mu(K_X+\Delta)$ for some $\mu>0$ and $\Delta'$ is big. 
    \item  If $\Delta$ is big over $T$, then there exists a klt pair $(X,\Delta_0)$ such that $K_X+\Delta\sim_{\Rr,T}K_X+\Delta_0$ and $\Delta_0\geq A\geq 0$ for some $\mathbb R$-divisor $A$ that is ample over $T$. \qed
\end{enumerate} 
\end{lem}



We also introduce a more flexible version of the MMP, where each step of the MMP is a contraction of an extremal face rather than an extremal ray. To avoid confusion, we call such MMP as ``face-contracting MMP". Face-contracting MMP naturally appears in the proofs of the special termination (\textit{cf}. \cite{Fuj07a}). The work \cite{Kol21} also uses special face-contracting MMP in order to ensure that the scaling numbers strictly drop. We remark that projectivity and $\mathbb Q$-factoriality of the ambient variety are not automatically presumed in the face-contracting MMP.

\begin{defn}[Face-contracting MMP]\label{defn: face contracting mmp} 
Let $(X,\Delta)$ be a log pair. A \emph{(birational) face-contracting} $(K_X+\Delta)$-MMP is a sequence of birational contractions
\[
\begin{tikzcd}
(X,\Delta)=:(X_0,\Delta_0) \arrow[r,dashed,"\varphi_0"] & (X_1,\Delta_1) \arrow[r,dashed,"\varphi_1"] & \cdots \arrow[r,dashed] & (X_i,\Delta_i) \arrow[r,dashed,"\varphi_i"] & \cdots,
\end{tikzcd}
\]
where $(X_i,\Delta_i)$ is a log pair and $\Delta_{i+1}:=(\varphi_i)_* \Delta_i$ for each $i$, and each $\varphi_i$ fits into a diagram 
\[
\begin{tikzcd}[column sep=tiny, row sep=small]
 X_i \arrow[rr, dashed, "\varphi_i"] \arrow[rd, "f_i"'] & & X_{i+1}  \arrow[ld, "f_i^+"]\\ & Z_i &
\end{tikzcd}
\]
satisfying the following:
\begin{enumerate}
    \item $f_i$ and $f_i^+$ are birational fibrations,
    \item $f_i^+$ is small,
    \item $-(K_{X_i}+\Delta_i)$ is $f_i$-ample, and 
    \item $K_{X_{i+1}}+\Delta_{i+1}$ is  $f_i^+$-ample.
\end{enumerate}

\end{defn}

We show that every step of a face-contracting MMP is crepant equivalent to a sequence of divisorial contractions and flips. 

\begin{lem} \label{lem:lift MMP dlt}
Let $(X,\Delta)$ be a log canonical pair and let $\pi\colon (Y,D)\to (X,\Delta)$ be a birational morphism from a $\bQ$-factorial dlt pair such that $K_Y+D=\pi^*(K_X+\Delta)+E$ for some $\pi$-exceptional $\bR$-divisor $E\ge 0$. Then any face-contracting $(K_X+\Delta)$-MMP
\[
\begin{tikzcd}
(X,\Delta)=:(X_0,\Delta_0) \arrow[r,dashed,"\varphi_0"] & (X_1,\Delta_1) \arrow[r,dashed,"\varphi_1"] & \cdots \arrow[r,dashed] & (X_i,\Delta_i) \arrow[r,dashed,"\varphi_i"] & \cdots 
\end{tikzcd}
\]
can be lifted to a commutative diagram
\[
\begin{tikzcd}
(Y,D) =: (Y_0,D_0) \arrow[r,dashed,"\psi_0"] \arrow[d, shift left=5ex, "\pi_0=\pi"'] & (Y_1,D_1) \arrow[r,dashed,"\psi_1"] \arrow[d,"\pi_1"'] & \cdots \arrow[r,dashed] & (Y_i,D_i) \arrow[r,dashed,"\psi_i"] \arrow[d,"\pi_i"'] & \cdots \\
(X,\Delta) =: (X_0,\Delta_0) \arrow[r,dashed,"\varphi_0"] & (X_1,\Delta_1) \arrow[r,dashed,"\varphi_1"] & \cdots \arrow[r,dashed] & (X_i,\Delta_i) \arrow[r,dashed,"\varphi_i"] & \cdots 
\end{tikzcd}
\]
where each $(Y_i,D_i)$ is a $\bQ$-factorial dlt pair and each $\psi_i$ is a sequence of $(K_{Y_i}+D_i)$-MMP. Moreover, for all $i\ge 1$ we have $K_{Y_i}+D_i = \pi_i^*(K_{X_i}+\Delta_i)$.
\end{lem}

\begin{proof}
Let $f_i: X_i\rightarrow Z_i$ and $f_i^+: X_{i+1}\rightarrow Z_i$ be the birational fibrations associated to $\varphi_i$. We prove the lemma by induction on $i$. When $i=0$, we simply set $(Y_0,D_0):=(Y,D)$ and $\pi_0:=\pi$. For $i\ge 1$, suppose that $(Y_{i-1},D_{i-1})$ has been constructed. By \cite{KM98}*{Lemma 3.38} and \cite{BCHM}*{Lemma 3.6.6(4)}, $(X_i,\Delta_i)$ is the ample model over $Z_{i-1}$ of $(Y_{i-1},D_{i-1})$. In particular, $(Y_{i-1},D_{i-1})$ has a minimal model over $Z_{i-1}$. Since the pair $(Y_{i-1},D_{i-1})$ is $\bQ$-factorial dlt by induction hypothesis, by \cite{Bir12lcflip}*{Corollary 1.2 and Theorem 1.9} or \cite[Corollaries 1.8 and 2.9]{HX13}, we may run a $(K_{Y_{i-1}}+D_{i-1})$-MMP over $Z_{i-1}$ with scaling of an ample divisor, and the MMP terminates with a $\bQ$-factorial dlt minimal model over $Z_{i-1}$. This gives $\psi_{i-1}: (Y_{i-1},D_{i-1})\dashrightarrow (Y_i,D_i)$, where $K_{Y_i}+D_i$ is nef over $Z_{i-1}$. Then $(X_i,\Delta_i)$ is the ample model over $Z_{i-1}$ of $(Y_i,D_i)$, so there exists an induced birational morphism $\pi_i:Y_i\to X_i$ such that $K_{Y_i}+D_i=\pi_i^{*}(K_{X_i}+\Delta_i)$. 
This finishes the induction.
\end{proof}

We can also lift an MMP to the terminalization at the price of possibly changing the terminal boundary during the process.

\begin{defn}
A \emph{$\mathbb Q$-factorial terminalization} of a klt pair $(X,\Delta)$ is a projective crepant birational morphism $\pi\colon  (Y,\Delta_Y)\rightarrow (X,\Delta)$ such that  $(Y,\Delta_Y)$ is a $\mathbb Q$-factorial terminal pair. We also say that $(Y,\Delta_Y)$ is a \emph{$\mathbb Q$-factorial terminalization} of $(X,\Delta)$.
\end{defn}

\begin{lem} \label{lem:lift MMP terminal}
Let $(X,\Delta)$ be a klt pair and let $(Y,D)$ be a $\bQ$-factorial terminal pair with a projective birational morphism $\pi\colon Y\rightarrow X$ such that $K_Y+D= \pi^*(K_X+\Delta)+E$ for some $\pi$-exceptional $\mathbb R$-divisor $E\ge 0$. Then any face-contracting $(K_X+\Delta)$-MMP
\[
\begin{tikzcd}
(X,\Delta)=:(X_0,\Delta_0) \arrow[r,dashed,"\varphi_0"] & (X_1,\Delta_1) \arrow[r,dashed,"\varphi_1"] & \cdots \arrow[r,dashed] & (X_i,\Delta_i) \arrow[r,dashed,"\varphi_i"] & \cdots 
\end{tikzcd}
\]
can be lifted to a commutative diagram
\[
\begin{tikzcd}
(Y,D) =: (Y_0,D_0) \arrow[r,dashed,"\psi_0"] \arrow[d, shift left=5ex, "\pi_0=\pi"'] & (Y_1,D_1) \arrow[r,dashed,"\psi_1"] \arrow[d,"\pi_1"'] & \cdots \arrow[r,dashed] & (Y_i,D_i) \arrow[r,dashed,"\psi_i"] \arrow[d,"\pi_i"'] & \cdots \\
(X,\Delta) =: (X_0,\Delta_0) \arrow[r,dashed,"\varphi_0"] & (X_1,\Delta_1) \arrow[r,dashed,"\varphi_1"] & \cdots \arrow[r,dashed] & (X_i,\Delta_i) \arrow[r,dashed,"\varphi_i"] & \cdots 
\end{tikzcd}
\]
where each $\psi_i$ is a sequence of $(K_{Y_i}+B_i)$-MMP with $0\le B_i\le D_i$ defined by
\begin{equation} \label{eq:B_i}
    \mult_F B_i = \max\{0,-a(F,X_{i+1},\Delta_{i+1})\}
\end{equation}
for all prime divisors $F$ on $Y_i$, the map $\psi_i$ does not contract any component of $B_i$, and $D_{i+1}=(\psi_i)_* B_i$. Moreover, $\pi_i: (Y_i,D_i)\rightarrow (X_i,\Delta_i)$ is a $\bQ$-factorial terminalization for every $i\ge 1$.
\end{lem}

\begin{proof}
The proof is similar to that of Lemma \ref{lem:lift MMP dlt} and proceeds by induction on $i$. Let $f_i: X_i\rightarrow Z_i$ and $f_i^+: X_{i+1}\rightarrow Z_i$ be the associated birational fibrations of $\varphi_i$. When $i=0$, set $(Y_0,D_0):=(Y,D)$ and $\pi_0:=\pi$. For $i\ge 1$, suppose that $(Y_{i-1},D_{i-1})$ has been constructed. Define $B_{i-1}$ by \eqref{eq:B_i}. Clearly $B_{i-1}\ge 0$. By the negativity lemma and the induction hypothesis that $K_{Y_{i-1}}+D_{i-1}\ge \pi_{i-1}^*(K_{i-1}+\Delta_{i-1})$, we have
\[
a(F,Y_{i-1},D_{i-1})\le a(F,X_{i-1},\Delta_{i-1})\le a(F,X_i,\Delta_i)
\]
for all divisors $F$ over $X$, hence $B_{i-1}\le D_{i-1}$. By \cite{KM98}*{Lemma 3.38} and \cite{BCHM}*{Lemma 3.6.6(4)}, $(X_i,\Delta_i)$ is the ample model over $Z_{i-1}$ of $(Y_{i-1},B_{i-1})$. Thus as in the proof of Lemma \ref{lem:lift MMP dlt}, by \cite{Bir12lcflip}*{Theorem 1.9} we may run a $(K_{Y_{i-1}}+B_{i-1})$-MMP $\psi_{i-1}: (Y_{i-1},B_{i-1})\dashrightarrow (Y_i,D_i)$ over $Z_{i-1}$, where $D_i:=(\psi_{i-1})_{*} B_{i-1}$, such that $K_{Y_i}+D_i$ is nef over $Z_{i-1}$ and $(X_i,\Delta_i)$ is the ample model over $Z_{i-1}$ of $(Y_i,D_i)$, which gives the morphism $\pi_i$. If $F$ is a component of $B_{i-1}$, then we have
\[
\mult_F B_{i-1}=-a(F,X_{i},\Delta_{i})=-a(F,Y_{i},D_{i})\le -a(F,Y_{i-1},B_{i-1})=\mult_F B_{i-1},
\]
thus the inequality in the middle is an equality; in particular, $F$ is not contracted by the MMP. By \cite{KM98}*{Corollary 3.43}, $(Y_i,D_i)$ is terminal. We finish the induction.
\end{proof}

The following corollary of the negativity lemma is needed in the proof of special termination in Section \ref{sec:tof dim 5}. 


\begin{lem} \label{lem:one step MMP isom along div}
Consider a commutative diagram
\[
\begin{tikzcd}[column sep=tiny, row sep=small]
 X \arrow[rr, dashed, "\varphi"] \arrow[rd, "f"'] & & X^+ \arrow[ld, "f^+"]\\ & Z &
\end{tikzcd}
\]
where $f,f^+$ are birational fibrations. Let $S$ be a prime divisor on $X$ such that $\varphi$ is an isomorphism at the generic point of $S$, and let $S^+$ be the strict transform of $S$ on $X^+$. Let $\Delta\ge 0$ be an $\bR$-divisor on $X$ such that $(X,S+\Delta)$ is dlt, and let $\Delta^+\ge 0$ be an $\bR$-divisor on $X^+$. 
Assume that
\begin{enumerate}
    \item $f_*\Delta = f^+_*\Delta^+$,
    \item $-(K_X+S+\Delta)$ is $f$-ample,
    \item $K_{X^+}+S^++\Delta^+$ is $f^+$-ample, and
    \item the induced birational map $(S,\Diff_S(\Delta))\dashrightarrow (S^+,\Diff_{S^+}(\Delta^+))$ extends to an isomorphism of pairs.
\end{enumerate}
Then $\varphi$ extends to an isomorphism around $S$.
\end{lem}

\begin{proof}
Let $g\colon Y\rightarrow X$ and $h\colon Y\rightarrow X^+$ be a common resolution such that $h=\varphi\circ g$. Let $\psi\colon Y\to Z$ be the induced morphism, i.e. $\psi:=f\circ g=f^+\circ h$. Write
\[
g^*(K_X+S+\Delta)=h^*(K_{X^+}+S^++\Delta^+)+E
\]
for some $\mathbb R$-divisor $E$. Let $S_Y:=g^{-1}_*S$, then we also have $S_Y=h^{-1}_*S^+$, hence $S_Y\not\subseteq \Supp(E)$. By adjunction, we get
\[
g|_{S_Y}^*(K_S+\Diff_S(\Delta))=h|_{S_Y}^*(K_{S^+}+\Diff_{S^+}(\Delta^+))+E|_{S_Y}.
\]
By assumption (4), we also have 
\[
g|_{S_Y}^*(K_S+\Diff_S(\Delta)) = h|_{S_Y}^*(K_{S^+}+\Diff_{S^+}(\Delta^+)).
\]
Combined with the previous equality we see that $E|_{S_Y}=0$.  By assumption (1), the $\bR$-divisor $E$ is exceptional over $Z$. By assumptions (2) and (3), $-E$ is nef over $Z$, hence by \cite[Lemma 3.39]{KM98}, we have $E\geq 0$, and for any closed point $z\in Z$, either $\psi^{-1}(z)\cap\Supp(E)=\emptyset$ or $\psi^{-1}(z)\subseteq\Supp(E)$. Since $E|_{S_Y}=0$, this implies that 
\[
\psi(\Supp(E))\cap \psi(S_Y)=\emptyset.
\]

Therefore, for any prime divisor $D$ over $Z$ with $\Center_Z(D)\cap \psi(S_Y)\neq\emptyset$ we have
\begin{equation} \label{eq:crepant near S}
    a(D,X,S+\Delta)=a(D,X^+,S^++\Delta^+).
\end{equation}
    
If $f$ is not an isomorphism around $S$, then there exists a prime divisor $D$ over $X$ such that $\Center_X (D)\cap S\neq\emptyset$ and $\Center_X (D)\subseteq\Ex(f)$. In particular, 
\[
\Center_Z(D)\cap \psi(S_Y) = f(\Center_X (D)) \cap f(S)\neq\emptyset.
\]
The negativity lemma \cite[Lemma 3.38]{KM98} and assumptions (2) and (3) then gives 
\[
a(D,X,S+\Delta)<a(D,X^+,S^++\Delta^+),
\]
which contradicts \eqref{eq:crepant near S}. We get a similar contradiction if $f^+$ is not an isomorphism around $S^+$. Thus $f$ and $f^+$ are isomorphisms around $S$ and $S^+$ respectively, and hence $\varphi$ is an isomorphism near $S$.
\end{proof}

We also need the termination of $4$-fold terminal MMP.

\begin{lem} \label{lem: 4fold terminal mmp terminates}
Let $(X,\Delta)$ be a $\mathbb Q$-factorial terminal pair of dimension $\le 4$. Then every $(K_X+\Delta)$-MMP
\[
\begin{tikzcd}
(X,\Delta)=:(X_0,\Delta_0) \arrow[r,dashed,"\varphi_0"] & (X_1,\Delta_1) \arrow[r,dashed,"\varphi_1"] & \cdots \arrow[r,dashed] & (X_i,\Delta_i) \arrow[r,dashed,"\varphi_i"] & \cdots 
\end{tikzcd}
\]
where $(X_i,\Delta_i)$ is terminal for all $i$, must terminate after finitely many steps.
\end{lem}

\begin{proof}
The same proof of \cites{Fuj04,Fuj05} in the projective setting for $\Qq$-divisors works here verbatim.
\end{proof}

\subsection{Explicit index bounds in dimension $3$}

In this subsection, we show that the Cartier index is uniformly bounded in $3$-dimensional MMP (without any general type or bigness assumption). For any klt pair $(X,\Delta)$, we set
\begin{equation} \label{eq:e(X,Delta)}
    e(X,\Delta):=\#\left\{ \mbox{Exceptional prime divisors $E$ over $X$ with $a(E,X,\Delta)\le 0$} \right\}.
\end{equation}
We also define $e_+(X,\Delta)$ as the sum of $e(X,\Delta)$ and the number of irreducible components of $\Supp(\Delta)$. By \cite[Proposition 2.36]{KM98}, $e(X,\Delta)<+\infty$.

\begin{lem} \label{lem:index klt 3-fold}
Let $(X,\Delta)$ be a klt pair of dimension $3$. Let $m\in\bN$ and $\varepsilon>0$. Assume that there are at most $m$ divisors $E$ over $X$ such that $a(E,X,\Delta)\in (0,\varepsilon)$. Then the Cartier index of any $\bQ$-Cartier Weil divisor on $X$ is at most
\[
\left(\left\lceil \frac{m+2}{\varepsilon}\right\rceil !\right)^{2^{e(X,\Delta)}} \left(\frac{3}{\mld(X,\Delta)} \right)^{2^{e(X,\Delta)}-1}.
\]
\end{lem}
\begin{proof}
This follows from the same proof of \cite{K+-flip-abundance}*{Theorem (6.11.5)}, replacing the base case (i.e. when $e(X,\Delta)=0$ and hence $(X,\Delta)$ is terminal) \cite{K+-flip-abundance}*{Lemma 6.9.7} of the induction by \cite{Sho96}*{Corollary 3.4}.
\end{proof}

\begin{cor} \label{cor:index 3-fold MMP}
Let $(X,\Delta)$ be a klt pair of dimension $3$. Let $\varepsilon$ be the smallest positive discrepancy of $(X,\Delta)$ (i.e. for any prime divisor $E$ over $X$ we have $a(E,X,\Delta)>0$ if and only if $a(E,X,\Delta)\ge \varepsilon$). Then in any $(K_X+\Delta)$-MMP, the Cartier index of any $\bQ$-Cartier Weil divisor is at most
\[
\left(\left\lceil \frac{e_+(X,\Delta)+3}{\varepsilon}\right\rceil !\right)^{2^{1+e_+(X,\Delta)}}.
\]
\end{cor}

\begin{proof}
Let $(X,\Delta)\dashrightarrow (X',\Delta')$ be a sequence of steps of a $(K_X+\Delta)$-MMP. We first prove that $\mld(X,\Delta)\ge \varepsilon$. In fact, by passing to the terminalization, it suffices to prove this when $(X,\Delta)$ is terminal. Write $\Delta=\sum b_i B_i$ where the $B_i$ are irreducible components of $\Supp(\Delta)$, then $\mld(X,\Delta)=1-\max_i\{b_i\}$. On the other hand, if we blow up a prime divisor in the smooth locus of $B_i$, we get an exceptional divisor $E$ with $a(E,X,\Delta)= 1-b_i>0$, thus as $\varepsilon$ is the smallest positive discrepancy we have $1-b_i\ge \varepsilon$. This gives $\mld(X,\Delta)\ge \varepsilon$ and hence we also have $\mld(X',\Delta')\ge \varepsilon$.

Let $e:=e_+(X,\Delta)$. By definition, every divisor $E$ over $X$ with $a(E,X,\Delta)<\varepsilon$ also satisfies $a(E,X,\Delta)\le 0$. As $(X,\Delta)\dashrightarrow (X',\Delta')$ is a $(K_X+\Delta)$-MMP, we have $a(E,X,\Delta)\le a(E,X',\Delta')$, hence any divisor $E$ with $a(E,X',\Delta')<\varepsilon$ also satisfies $a(E,X,\Delta)\le 0$. If $E$ is exceptional over $X'$, then either it is also exceptional over $X$, or it is a component of $\Supp(\Delta)$. It follows that there are at most $e$ exceptional divisors $E$ over $X'$ with $a(E,X',\Delta')<\varepsilon$; in particular, there are at most $e$ divisors $E$ with $0<a(E,X',\Delta')<\varepsilon$ (all such divisors are necessarily exceptional over $X'$). 
Since we also have $e(X',\Delta')\le e_+(X',\Delta')\le  e_+(X,\Delta)=e$ and $\mld(X',\Delta')\ge \varepsilon$, Lemma \ref{lem:index klt 3-fold} implies that the Cartier index of any $\bQ$-Cartier Weil divisor on $X'$ is at most
\[
\left(\left\lceil \frac{e+2}{\varepsilon}\right\rceil !\right)^{2^{e}}\cdot \left(\left\lceil \frac{3}{\varepsilon}\right\rceil !\right)^{2^{e-1}}\le \left(\left\lceil \frac{e+3}{\varepsilon}\right\rceil !\right)^{2^{e+1}}.
\]
\end{proof}

\section{Termination in low dimension}\label{sec:tof dim 5}


In this section, we prove some effective termination result for general type MMPs in dimensions at most $5$. In particular, we prove Theorem \ref{main:termination}.

\subsection{Termination}

We first prove a non-effective version of termination in dimensions four and five. 

\begin{thm} \label{thm:termination dim<=5}
Let $(X,\Delta)$ be a klt pair. Assume that $\Delta$ is big, and either
\begin{enumerate}
    \item $\dim X=4$, or
    \item $\dim X=5$ and $K_X+\Delta\sim_{\Rr} D\ge 0$ for some $\mathbb R$-divisor $D$.
\end{enumerate}
Then any $(K_X+\Delta)$-MMP terminates.
\end{thm}

Before proving this result, we give an immediate corollary.

\begin{cor}\label{cor: termination of mmp in dim 5 for gt}
Let $(X,\Delta)$ be a klt pair of dimension at most $5$ and $\pi: X\rightarrow T$ a projective morphism. Assume that at least one of the following holds.
\begin{enumerate}
    \item $K_X+\Delta$ is big over $T$.
    \item $K_X+\Delta$ is pseudo-effective over $T$ and $\Delta$ is big over $T$.
    \item $\dim X = 4$ and $\Delta$ is big over $T$.
\end{enumerate}
Then any $(K_X+\Delta)$-MMP over $T$ terminates.
\end{cor}

\begin{proof}
By Lemma \ref{lem: gt klt sim to big boundary}, we may assume that $\Delta$ is big in all cases. 
Then (3) follows from Theorem \ref{thm:termination dim<=5}(1), and (1) follows from (2), so we only need to prove (2) when $\dim X=5$.

By \cite{BCHM}*{Theorem D}, there exists an $\Rr$-divisor $D\ge0$ on $X$, such that $K_X+\Delta\sim_{\Rr,T}D$, or equivalently, there exists an $\Rr$-divisor $G_T$ on $T$ such that $K_X+\Delta\sim_{\Rr} D+\pi^{*}G_T$. Let $H_T\ge 0$ be an ample divisor on $T$ such that $H_T+G_T\ge 0$. By Bertini's theorem, there exists an $\Rr$-divisor $0\le \Delta'\sim_\bR \Delta+\pi^*H_T$ on $X$ such that $(X,\Delta')$ is klt. Then $K_X+\Delta'\sim_{\Rr} K_X+\Delta+\pi^{*}H_T\sim_{\Rr} D+\pi^{*}(G_T+H_T)\ge 0$ and any $(K_X+\Delta)$-MMP over $T$ is a also a $(K_X+\Delta')$-MMP, thus we conclude (2) by Theorem \ref{thm:termination dim<=5}(2).
\end{proof}


In the rest of this subsection we prove Theorem \ref{thm:termination dim<=5}.

\begin{lem} \label{lem:index bdd imply mld discrete}
Let $n$ and $r$ be two positive integers and $S_0\subseteq [0,1]$ a finite set. Then there exists a discrete set $S\subseteq \bR_{\ge 0}$ depending only on $n,r$ and $S_0$ satisfying the following. Assume that
\begin{enumerate}
    \item $(X,\Delta)$ is a dlt pair of dimension $n$,
    \item $\Coef(\Delta)\subseteq S_0$,
    \item $rL$ is Cartier for any $\mathbb Q$-Cartier Weil divisor $L$ on $X$, and
    \item $W$ is a log canonical center of $(X,\Delta)$ of dimension $\geq 1$ and $(W,\Delta_W)$ is the dlt pair induced by adjunction
    $K_W+\Delta_W:=(K_X+\Delta)|_{W}.$
\end{enumerate}
Then for any prime divisor $E$ over $W$, we have $A_{W,\Delta_W}(E)\in S$.
\end{lem}

\begin{proof}
By  \cite{HLS24}*{Theorem 5.6}, there exist real numbers $a_1,\dots,a_k\in (0,1]$ and a positive integer $N$ depending only on $n$ and $S_0$, and $\bQ$-divisors $\Delta_1,\dots,\Delta_k\ge 0$ such that $\sum_{i=1}^k a_i=1$, $\Delta=\sum_{i=1}^k a_i \Delta_i$, each pair $(X,\Delta_i)$ is dlt, and $N\Delta_i$ is a Weil divisor for each $i$. In particular, $W$ is a log canonical center of each $(X,\Delta_i)$. Let $K_W+\Delta_{W,i}:=(K_X+\Delta_i)|_W$ for each $i$, then each $(W,\Delta_{W,i})$ is dlt and $\Delta_W=\sum_{i=1}^ka_i\Delta_{W,i}$. For each $i$, since $Nr(K_X+\Delta_i)$ is Cartier by assumption (3), we have $Nr(K_W+\Delta_{W,i})$ is also Cartier. Thus the set
$$S=\frac{1}{Nr}\left\{\sum_{i=1}^k m_ia_i\middle|\, m_i\in\mathbb N\right\}$$
satisfies the statement of the lemma.
\end{proof}

We first consider the 4-fold case of Theorem \ref{thm:termination dim<=5}. In view of Theorem \ref{thm:index bdd} and the above Lemma \ref{lem:index bdd imply mld discrete}, the key is to prove the following statement.

\begin{lem} \label{lem:termination dim 4 mld discrete}
Let $(X,\Delta)$ be a klt pair of dimension at most $4$ and let $S\subseteq \bR$ be a discrete set. Let  
\[
\begin{tikzcd}
(X,\Delta)=:(X_0,\Delta_0) \arrow[r,dashed,"\varphi_0"] & (X_1,\Delta_1) \arrow[r,dashed,"\varphi_1"] & \dots \arrow[r,dashed] & (X_i,\Delta_i) \arrow[r,dashed,"\varphi_i"] & \dots 
\end{tikzcd}
\]
be a face-contracting $(K_X+\Delta)$-MMP (Definition \ref{defn: face contracting mmp}) such that $a(E,X_i,\Delta_i)\in S$ for all $i$ and all prime divisors $E$ over $X$. Then this face-contracting MMP terminates.
\end{lem}

\begin{proof}
Since $(X_i,\Delta_i)$ is klt, by \cite[Proposition 2.36]{KM98} there are only finitely many prime divisors $E$ over $X$ with $a(E,X_i,\Delta_i)\le 0$. The number $l_i$ of such divisors is non-increasing in the MMP, thus after possibly truncating the MMP we may assume that this number stabilizes, i.e. $l_0=l_1=\cdots=l$ for some $l\in \bN$. Let 
$E_1,\dots,E_l$ be the corresponding prime divisors over $X$ with $a(E_j,X_i,\Delta_i)\le 0$ for all $i$ and $j$. 

Since $S$ is discrete, $S\cap (-1,0]$ is a finite set. Thus as $a(E_j,X_i,\Delta_i)\in S\cap (-1,0]$ by assumption and $a(E_j,X_{i+1},\Delta_{i+1})\ge a(E_j,X_i,\Delta_i)$ for all $i$ and $j$, after another truncation of the MMP we may assume that 
\[
a(E_j,X_0,\Delta_0)=a(E_j,X_1,\Delta_1)=\cdots=a(E_j,X_i,\Delta_i)=\cdots
\]
for all $j=1,\cdots,l$. Let $\pi: (Y,D)\rightarrow (X,\Delta)$ be a $\bQ$-factorial terminalization of $(X,\Delta)$. By Lemma \ref{lem:lift MMP terminal}, there exists a sequence of birational maps
\[
\begin{tikzcd}
(Y,D)=:(Y_0,D_0) \arrow[r,dashed,"\psi_0"] & (Y_1,D_1) \arrow[r,dashed,"\psi_1"] & \cdots \arrow[r,dashed] & (Y_i,D_i) \arrow[r,dashed,"\psi_i"] & \cdots
\end{tikzcd}
\]
such that each $(Y_i,D_i)$ is a $\bQ$-factorial terminalization of $(X_i,\Delta_i)$ and each $\psi_i$ is composed of a sequence of steps of a 
$(K_{Y_i}+D_i)$-MMP. As $\dim X\le 4$, this sequence terminates by Lemma \ref{lem: 4fold terminal mmp terminates}, hence the original face-contracting $(K_X+\Delta)$-MMP also terminates.
\end{proof}

\begin{proof}[Proof of Theorem \ref{thm:termination dim<=5} when $\dim X=4$]
By Theorem \ref{thm:index bdd}, the Cartier index of $\bQ$-Cartier Weil divisors is uniformly bounded in the MMP, thus by Lemma \ref{lem:index bdd imply mld discrete}, the discrepancy of primes divisors over $X$ takes values in some discrete set in the MMP. We then conclude by Lemma \ref{lem:termination dim 4 mld discrete}. 
\end{proof}

We move on to the $5$-fold case of Theorem \ref{thm:termination dim<=5}. A key observation is that since the Cariter index of $\mathbb Q$-Cartier Weil divisors are bounded from above by Theorem \ref{thm:index bdd}, in Birkar's inductive approach to termination \cite{Bir07}, we only need termination of flips in lower dimensions when the index is bounded.

\begin{lem} \label{lem:special termination dim 5 index bdd}
Let $(Y,D)$ be a $\bQ$-factorial dlt pair of dimension $5$. Let
\begin{equation} \label{eq:dlt MMP}
\begin{tikzcd}
(Y,D)=:(Y_0,D_0) \arrow[r,dashed,"\psi_0"] & (Y_1,D_1) \arrow[r,dashed,"\psi_1"] & \cdots \arrow[r,dashed] & (Y_i,D_i) \arrow[r,dashed,"\psi_i"] & \cdots
\end{tikzcd}
\end{equation}
be a 
$(K_Y+D)$-MMP. Assume that there exists some integer $r\in \bN$ such that the Cartier index of every $\bQ$-Cartier Weil divisor on $Y_i$ is at most $r$ for every $i$. Then $\psi_i$ is an isomorphism near $\lfloor D_i \rfloor$ for any $i\gg 0$.
\end{lem}

\begin{rem}
As will be clear from the proof, if Lemma \ref{lem:termination dim 4 mld discrete} holds in dimension $n-1$, then the above lemma holds in dimension $n$.
\end{rem}

\begin{proof}
We only outline the main steps since this essentially follows from the same proof of the special termination theorem, see e.g. \cite{Fuj07a}*{Proof of Theorem 4.2.1}. First note that by the negativity lemma as in Step 1 of \textit{loc. cit.}, after truncating the MMP we may assume that it is an isomorphism at the generic point of every log canonical center of $(Y,D)$. Let $W\subseteq \Supp\lfloor D\rfloor$ be a log canonical center of $(Y,D)$. Let $W_i$ be its strict transform on $Y_i$ and define $B_i$ by adjunction $(K_{Y_i}+D_i)|_{W_i}=K_{W_i}+B_i$. We next prove by induction on $\dim W$ that the induced birational map $\varphi_i\colon W_i\dashrightarrow W_{i+1}$ induces an isomorphism of pairs $(W_i,B_i)\cong (W_{i+1},B_{i+1})$ when $i\gg 0$. The base case $\dim W=0$ is clear, hence we may assume that $\dim W\ge 1$. Let $B=B_0$. By Lemma \ref{lem:one step MMP isom along div} and the induction hypothesis, we see that the birational map $W_i\dashrightarrow W_{i+1}$ is an isomorphism near $\lfloor B_i \rfloor$ for any $i\gg 0$. By \cite{Fuj07a}*{Proposition 4.2.14 and Lemma 4.2.15}, after another truncation of the MMP we may assume the induced sequence of birational maps
\[
\begin{tikzcd}
(W_0,B_0) \arrow[r,dashed, "\varphi_0"] & (W_1,B_1) \arrow[r,dashed, "\varphi_1"] & \cdots \arrow[r,dashed] & (W_i,B_i) \arrow[r,dashed,"\varphi_i"] & \cdots
\end{tikzcd}
\]
is a face-contracting $(K_W+B)$-MMP. Let $U=W\setminus \lfloor B \rfloor$ and $B_U=B|_U$. Then we also get a sequence of $(K_U+B_U)$-MMP
\begin{equation} \label{eq:MMP on open log canonical center}
\begin{tikzcd}
(U_0,B_{U_0}) \arrow[r,dashed] & (U_1,B_{U_1}) \arrow[r,dashed] & \cdots \arrow[r,dashed] & (U_i,B_{U_i}) \arrow[r,dashed] & \cdots
\end{tikzcd}
\end{equation}
where $U_i=W_i\setminus \lfloor B_i \rfloor$ and $B_{U_i}=B_i|_{U_i}$. Note that $(U,B_U)$ is klt and $\dim U\le 4$. Since the index of every $\bQ$-Cartier Weil divisor on $Y_i$ is at most $r$ for every $i$, by Lemmas \ref{lem:index bdd imply mld discrete} and \ref{lem:termination dim 4 mld discrete} we deduce that \eqref{eq:MMP on open log canonical center} becomes an isomorphism when $i\gg 0$. This completes the induction. 
Since every component of $\lfloor D_i \rfloor$ is a log canonical center of $(Y_i,D_i)$, the statement now follows from another application of Lemma \ref{lem:one step MMP isom along div}.
\end{proof}

To apply Lemma \ref{lem:special termination dim 5 index bdd}, we need to verify the uniform boundedness of the Cartier index when we lift an MMP to some dlt modification.


\begin{lem} \label{lem:dlt lift index bdd}
Let $(X,B)$ be a log canonical pair and let $\pi\colon (Y,D)\to (X,B)$ be a crepant birational morphism such that $\mult_E D>0$ for all $\pi$-exceptional divisor $E$. Assume that there exists a big $\bR$-divisor $\Delta\ge 0$ on $X$ such that $(X,\Delta)$ is klt and $K_X+B\sim_\bR \mu(K_X+\Delta)$ for some $\mu>0$. Then there exists a positive integer $r$, such that for any MMP type contraction $\psi\colon (Y,D)\dashrightarrow (Y',D')$ with $D'=\psi_* D$, the Cartier index of any $\bQ$-Cartier Weil divisor on $Y'$ is at most $r$.
\end{lem}

\begin{proof}
Since $\Delta$ is big, after possibly replacing $\Delta$ by some other $\bR$-divisor $0\le \Delta_0\sim_\bR \Delta$ we may assume by Lemma \ref{lem: gt klt sim to big boundary} that $\Delta\ge A$ where 
$A\geq 0$ is an ample $\mathbb R$-divisor. Since $(X,\Delta)$ is klt and $\mult_E D>0$ for any $\pi$-exceptional prime divisor $E$, we can choose some $0<\varepsilon\ll 1$ such that
\begin{itemize}
    \item the pair $(X,\Delta':=(1-\varepsilon)B+\varepsilon\Delta)$ is klt, and 
    \item $\mult_E G>0$ for any $\pi$-exceptional prime divisor $E$, where the $\bR$-divisor $G$ is defined by crepant pullback: $K_Y+G=\pi^*(K_X+\Delta')$.
\end{itemize}
In particular, $\Supp(\pi^*A)\subseteq \Supp(G)$ and hence $G$ is big. We have
\[
K_Y+G = (1-\varepsilon)\pi^*(K_X+B) + \varepsilon \pi^*(K_X+\Delta) \sim_\bR (1-\varepsilon+\varepsilon\mu^{-1}) (K_Y+D).
\]
Thus as $\psi\colon (Y,D)\dashrightarrow (Y',D')$ is an MMP type contraction, by \cite[Lemma 2.13]{HQZ25} it also induces an MMP type contraction $(Y,G)\dashrightarrow (Y,G')$ with $G'=\psi_* G$. We then conclude by Theorem \ref{thm:index bdd}.
\end{proof}

\begin{lem} \label{lem:special termination dim 5 log canonical}
Let $(X,\Delta)$ be a $\mathbb Q$-factorial klt pair of dimension $5$ and let $B$ be an $\bR$-divisor on $X$ such that $(X,B)$ is log canonical. Assume that $\Delta$ is big and $K_X+B\sim_\bR \mu(K_X+\Delta)$ for some $\mu>0$. Let 
\begin{equation} \label{eq:Q-factorial log canonical MMP}
\begin{tikzcd}
(X,B)=:(X_0,B_0) \arrow[r,dashed,"\varphi_0"] & (X_1,B_1) \arrow[r,dashed,"\varphi_1"] & \dots \arrow[r,dashed] & (X_i,B_i) \arrow[r,dashed,"\varphi_i"] & \dots 
\end{tikzcd}
\end{equation}
be a $(K_X+B)$-MMP. Then $\varphi_i$ is an isomorphism near $\Nklt(X_i,B_i)$ for any $i\gg 0$.
\end{lem}

\begin{proof}
By \cite[Corollary 1.36]{Kol13}, there exists a $\bQ$-factorial dlt pair $(Y,D)$ and a crepant birational morphism $\pi\colon (Y,D)\to (X,B)$ such that every $\pi$-exceptional divisor has coefficient $1$ in $D$. By Lemma \ref{lem:lift MMP dlt}, we can lift \eqref{eq:Q-factorial log canonical MMP} to a sequence
\[
\begin{tikzcd}
(Y,D)=:(Y_0,D_0) \arrow[r,dashed,"\psi_0"] & (Y_1,D_1) \arrow[r,dashed,"\psi_1"] & \cdots \arrow[r,dashed] & (Y_i,D_i) \arrow[r,dashed,"\psi_i"] & \cdots
\end{tikzcd}
\]
such that each $\psi_i$ is a sequence of steps of a $(K_{Y_i}+D_i)$-MMP, and there are crepant birational morphisms $\pi_i: (Y_i,D_i)\to (X_i,B_i)$. 
We claim that $\Ex(\pi_i)\subseteq\lfloor D_i \rfloor$ for all $i$. In fact, since $X_i$ is $\bQ$-factorial, $\Ex(\pi_i)$ is of pure codimension one in $Y_i$ by \cite[Corollary 2.63]{KM98}. In particular, the claim is clear when $i=0$, so we may assume that $i\ge 1$. For any irreducible component $E$ of $\Ex(\pi_i)$, we have
\[
a(E,X_i,\Delta_i)= -\mult_E D_i = -\mult_E D_{i-1}=a(E,X_{i-1},\Delta_{i-1}),
\]
so $E$ cannot be the strict transform of some $\varphi_{i-1}$-exceptional divisor (otherwise we get a contradiction to the negativity lemma). It follows that the strict transform of $E$ on $Y_{i-1}$ is contained in $\Ex(\pi_{i-1})$, hence $\Ex(\pi_i) = (\psi_{i-1})_* \Ex(\pi_{i-1})$, and the claim follows by induction on $i$. 

By Lemma \ref{lem:dlt lift index bdd}, there exists some positive integer $r$ such that the Cartier index of any $\bQ$-Cartier Weil divisor on $Y_i$ is at most $r$ for all $i$. Thus by Lemma \ref{lem:special termination dim 5 index bdd}, we see that $\psi_i$ is an isomorphism near $\lfloor D_i\rfloor$ for any $i\gg 0$. Since $(Y_i,D_i)$ is dlt and $\pi_i$ is crepant, combined with the previous claim $\Ex(\pi_i)\subseteq\lfloor D_i \rfloor$ we see that $\lfloor D_i \rfloor=\pi_i^{-1}(\Nklt(X_i,B_i))$. Therefore, $\varphi_i$ is also an isomorphism near $\Nklt(X_i,B_i)$ for any $i\gg 0$; otherwise, by the negativity lemma, for some $i\gg 0$ we can find a prime divisor $E$ over $X$ whose center is contained in $\Nklt(X_i,B_i)$ such that
\[
a(E,Y_i,D_i)=a(E,X_i,B_i)<a(E,X_{i+1},B_{i+1})=a(E,Y_{i+1},D_{i+1}),
\]
which is a contradiction as the center of $E$ on $Y_i$ is contained in $\lfloor D_i \rfloor$ and $\psi_i$ is an isomorphism around $\lfloor D_i \rfloor$.
\end{proof}

\begin{proof}[Proof of Theorem \ref{thm:termination dim<=5} when $\dim X=5$]
Let \[
\begin{tikzcd}
(X,\Delta)=:(X_0,\Delta_0) \arrow[r,dashed,"\varphi_0"] & (X_1,\Delta_1) \arrow[r,dashed,"\varphi_1"] & \dots \arrow[r,dashed] & (X_i,\Delta_i) \arrow[r,dashed,"\varphi_i"] & \dots 
\end{tikzcd}
\]
be a $(K_X+\Delta)$-MMP. Let $D_i$ be the strict transform of $D$ on $X_i$, and let $t_0=\lct(X,\Delta;D)$. Then $(X,\Delta+t_0 D)$ is log canonical and $K_X+\Delta+t_0 D\sim_\bR (1+t_0)(K_X+\Delta)$. Since $K_X+\Delta\sim_\bR D$, the MMP is an isomorphism away from $\Supp(D_i)$ and any $(K_X+\Delta)$-MMP is also a $(K_X+\Delta+t_0 D)$-MMP. By Lemma \ref{lem:special termination dim 5 log canonical}, after finitely many steps, the MMP is an isomorphism near $\Nklt(X_i,\Delta_i+t_0 D_i)$. In other words, after truncation it induces a sequence of $(K_U+\Delta_U)$-MMP where $U=X\setminus \Nklt(X,\Delta+t_0 D)$ and $\Delta_U=\Delta|_U$. Now $(U,\Delta_U+t_0 D_U)$ is klt (where $D_U=D|_U$) and hence $t_1:=\lct(U,\Delta_U;D_U)>t_0$. Replacing $(X,\Delta)$, $D$ with $(U,\Delta_U)$, $D_U$ and repeat the same argument, we see that the MMP is again an isomorphism near $\Nklt(X_i,\Delta_i+t_1 D_i)$ after finitely many steps, and thus we can inductively construct, if the MMP does not terminate, a strictly increasing sequence $t_0<t_1<\dots$ of log canonical thresholds, where each $t_i$ has the form $\lct_{x_j}(X_j,\Delta_j;D_j)$ for some $j\in\bN$ and some point $x_j\in X_j$. This contradicts the ACC of log canonical thresholds \cite{HMX-ACC}*{Theorem 1.1}. It follows that the MMP must terminate. 
\end{proof}

\subsection{Effective termination}

We next upgrade the termination result from the previous subsection to an effective one and give some further applications.

\begin{thm}\label{thm: eff termination for gt}
Let $(X,\Delta)$ be a klt pair of dimension $\leq 5$ and $X\to T$ a projective morphism. Assume that at least one of the following holds.
\begin{enumerate}
    \item $K_X+\Delta$ is big over $T$,
    \item $K_X+\Delta$ is pseudo-effective over $T$, and $\Delta$ is big over $T$, or
    \item $\dim X\le 4$, and $\Delta$ is big over $T$.
\end{enumerate}
Then there are finitely many birational maps $\phi_i: X\dashrightarrow Y_i$ over $T$, $1\le i\le m$, such that for any sequences of steps of a $(K_X+\Delta)$-MMP $\phi: X\dashrightarrow Y$ over $T$, we have $\phi=\phi_i$ for some $i$. In particular, any $(K_X+\Delta)$-MMP over $T$ terminates in at most $m$ steps. 
\end{thm}

This essentially follows from Theorem \ref{thm:termination dim<=5} and K\H{o}nig's lemma: there are only finitely many extremal rays we can possibly contract at each step of the MMP, so if the MMP can have an arbitrary large number of steps, we would be able to construct an MMP that does not terminate, which contradicts Corollary \ref{cor: termination of mmp in dim 5 for gt}. We spell out the details in the next two lemmas.

\begin{lem}\label{lem: possible of one-step mmp}
Let $(X,\Delta)$ be a $\mathbb Q$-factorial klt pair and $X\to T$ a projective morphism such that $K_X+\Delta$ or $\Delta$ is big over $T$. Then there are only finitely many possible $(K_X+\Delta)$-divisorial contractions or flips over $T$.
\end{lem}

\begin{proof}
By Lemma \ref{lem: gt klt sim to big boundary}, we may assume that $\Delta=\Delta_0+A$, where $\Delta_0\ge 0$ and $A\geq 0$ is an ample $\mathbb R$-divisor. By the cone theorem \cite[Theorem 3.25(2)]{KM98}, there are only finitely many $(K_X+\Delta)$-negative rays over $T$. In particular, there are only finitely many possible $(K_X+\Delta)$-divisorial contractions or flips over $T$. 
\end{proof}

\begin{lem}\label{lem: temrination for gt imply effective termination}
Let $(X,\Delta)$ be a $\mathbb Q$-factorial klt pair and $X\to T$ a projective morphism such that $K_X+\Delta$ or $\Delta$ is big over $T$. Assume that any $(K_X+\Delta)$-MMP over $T$ terminates. Then there are finitely many birational maps $\phi_i: X\dashrightarrow Y_i$ over $T$, $1\le i\le m$, such that for any sequences of steps of a $(K_X+\Delta)$-MMP $\phi: X\dashrightarrow Y$ over $T$, we have $\phi=\phi_i$ for some $i$. In particular, any $(K_X+\Delta)$-MMP over $T$ terminates in at most $m$ steps. 
\end{lem}

\begin{proof}
By Lemma \ref{lem: possible of one-step mmp}, there are only finitely many possible $(K_X+\Delta)$-divisorial contractions or flips $\phi_i:(X_0,\Delta_0):=(X,\Delta)\dashrightarrow (X_1^{(j)},\Delta_1^{(j)})$, $1\le j\le m_1$, where $\Delta_1^{(j)}$ is the strict transform of $\Delta_0$ on $X_1^{(j)}$. If the lemma does not hold for $(X,\Delta)$, then the lemma does not hold for $(X_1^{(j)},\Delta_1^{(j)})$ for some $j$. Let $(X_1,\Delta_1):=(X_1^{(j)},\Delta_1^{(j)})$ and repeat the same argument with $(X_1,\Delta_1)$ in place of $(X_0,\Delta_0)$. Inductively, we may construct $(X_i,\Delta_i)$, $i\ge 1$, such that $$(X_0,\Delta_0)\dashrightarrow (X_1,\Delta_1) \dashrightarrow \cdots \dashrightarrow (X_i,\Delta_i)\dashrightarrow (X_{i+1},\Delta_{i+1})\dashrightarrow \cdots$$
is an infinite sequence of $(K_X+\Delta)$-MMP over $T$, and the lemma does not hold for each $(X_i,\Delta_i)$. This is impossible as any $(K_X+\Delta)$-MMP over $T$ terminates.
\end{proof} 

\begin{proof}[Proof of Theorem \ref{thm: eff termination for gt}]
This follows from Corollary \ref{cor: termination of mmp in dim 5 for gt} and Lemma \ref{lem: temrination for gt imply effective termination}.
\end{proof}

\begin{proof}[Proof of Theorem \ref{main:termination}]
This follows from Theorem \ref{thm: eff termination for gt}.
\end{proof}

As another application of the termination result in this section, we prove the strong Sarkisov program conjecture \cite{He23}*{Conjecture 3.2} in dimension $5$. We refer to \cite{He23} and the reference therein for more details of the strong Sarkisov program.
\begin{cor}
    The strong Sarkisov program holds in dimension $5$.
\end{cor}

\begin{proof}
    It follows from Corollary \ref{cor: termination of mmp in dim 5 for gt} and \cite{He23}*{Corollary 3.5}.
\end{proof}

\section{Explicit termination}\label{sec:explicit tof}

In this section, we give explicit bounds (see Corollary \ref{cor:terminal explicit}, Theorems \ref{thm:4fold explicit general type}, \ref{thm:4fold explicit boundary big}, and \ref{thm: threefold explcit upper bound}) on the number of steps in various types of MMP in dimension at most $4$. This improves Theorem \ref{thm: eff termination for gt}, which only gives the existence of an inexplicit upper bound.

\subsection{Invariants} \label{ss:explicit invariant}

Before we state the result, we need to introduce a few invariants that we shall use to formulate the explicit bound and recall some of their basic properties. 

\begin{defn}[Picard number]
Let $X$ be a smooth  variety. Let $j\colon X\hookrightarrow \overline X$ be a smooth projective compactification. We set $N^1(X):=\Pic(X)/j^*\Pic^0(\overline{X})$ and define the \emph{Picard number} of $X$ as
\[
\rho(X):=\dim N^1(X)_\bR.
\]
This is well-defined since $j^*\Pic^0(\overline{X})$ is independent of the compactification. In general, for any normal  variety $X$, we define $\rho(X):=\rho(U)$ where $U$ is the smooth locus of $X$.
\end{defn}

\begin{lem} \label{lem:compare rho}
Let $\pi\colon X\rightarrow Z$ be a projective birational morphism of normal varieties. Then $\rho(X)=\rho(Z)+\ell$, where $\ell$ is the number of $\pi$-exceptional divisors.
\end{lem}
\begin{proof}
If $X$ is smooth, this follows from \cite[Lemma 1.6(1)]{AHK07}. The general case can be reduced to this by considering a common resolution of $X$ and $Z$.
\end{proof}

\begin{defn}[Algebraic Betti number]
Let $X$ be a variety and let $m\in \bN$. We denote by $H^{\alg}_{2m}(X)$ the subspace in the singular homology group
$H_{2m}(X,\mathbb R)$ generated by the the classes of $m$-dimensional complete algebraic subvarieties. We denote by $h^{\alg}_{2m}(X)$ the dimension of $H^{\alg}_{2m}(X)$.
\end{defn}

\begin{lem}\label{lem: compare h^alg}
    Let $m\in\bN$ and let $\pi\colon X\rightarrow Z$ be a projective birational morphism of normal varieties such that $\dim \Ex(\pi)\leq m$. Then 
    \begin{equation} \label{eq:h^alg}
        h^{\alg}_{2m}(Z) + \min\{1,\ell\}\le h^{\alg}_{2m}(X) \le h^{\alg}_{2m}(Z) + \ell
    \end{equation}
    where $\ell$ is the number of $m$-dimensional complete irreducible components of $\Ex(\pi)$.
\end{lem}

\begin{proof}
For simplicity, we abbreviate all singular homology groups $H_*(-,\bR)$ as $H_*(-)$ in this proof.

As $\pi$ is projective, every complete subvariety of $Z$ is the image of some complete subvariety of $X$ of the same dimension. Thus the natural homomorphism $\pi_*\colon H_{2m}(X)\to H_{2m}(Z)$ induces a surjection $H^{\alg}_{2m}(X) \twoheadrightarrow H^{\alg}_{2m}(Z)$. In particular, $h^{\alg}_{2m}(Z)\le h^{\alg}_{2m}(X)$. For any complete subvariety $V\subseteq X$ of dimension $m$ and any projective compactification $\overline{X}$ of $X$, we have $[V]\neq 0\in H^{\alg}_{2m}(\overline{X})$, hence $[V]\neq 0\in H^{\alg}_{2m}(X)$ as well. If $\pi$ contracts at least one complete subvariety of dimension $m$ (i.e. $\ell \ge 1$), then it follows that $\ker(\pi_*)\neq 0$ and we get $h^{\alg}_{2m}(Z)\le h^{\alg}_{2m}(X)-1$. This proves the first inequality in \eqref{eq:h^alg}.

Let $W=\Ex(\pi)$. By assumption, $\dim \pi(W)\le m-1$, thus $H_i (\pi(W))=0$ for any $i>2m-2$. By the long exact sequence of homology we deduce that $H_{2m}(Z)\cong H_{2m}(Z,\pi(W))$ where $H_{2m}(Z,\pi(W))$ is the relative homology. Since $X\setminus W\cong Z\setminus \pi(W)$, by excision we also have $H_{2m}(Z,\pi(W))\cong H_{2m}(X,W)$. Combining with the long exact sequence of homology we see that
\[
H_{2m}(W)\to H_{2m}(X)\to H_{2m}(X,W)\cong H_{2m}(Z)
\]
is exact. Note that $\dim H_{2m}(W) = \ell$ by the Mayer–Vietoris sequence and the fact that for any variety $V$ of dimension at most $m$, we have 
\[
\dim H_{2m}(V) = \begin{cases}
    1, & \mbox{if $V$ is complete and } \dim V = m, \\
    0, & \mbox{otherwise}.
\end{cases}
\]
It follows that $\ker(\pi_*)\le \ell$
and hence $h^{\alg}_{2m}(X) \le h^{\alg}_{2m}(Z) + \ell$. This proves the other inequality in \eqref{eq:h^alg}.
\end{proof}

\begin{defn} \label{defn:rho,d,s}
Let $(X,\Delta)$ be a klt pair of dimension $n$ and let $\widetilde{B}$ be the normalization of $B:=\Supp(\Delta)$. We define
\begin{equation} \label{eq:rho(X,Delta)}
\rho(X,\Delta):=\rho(X)+\rho(\widetilde{B})+h^{\mathrm{alg}}_{2n-4}(X).
\end{equation}
Following \cite{AHK07}, we call an exceptional divisor $E$ over $X$ with $a(E,X,\Delta)<1$ an \textit{echo} of $(X,\Delta)$ if its center $c_X(E)$ is a codimension one point in the smooth locus of $B$. We also define
\begin{equation} \label{eq:d(X,Delta)}
    d(X,\Delta):=\#\left\{ 
    \begin{tabular}{c}
         \mbox{Exceptional prime divisor $E$ over $X$ with}  \\
         \mbox{$a(E,X,\Delta)<1$ that is not an echo of $(X,\Delta)$}
    \end{tabular}
    \right\}.
\end{equation}
Note that if $(X,\Delta)$ is terminal, then $d(X,\Delta)<+\infty$ by \cite{K+-flip-abundance}*{Lemma (4.14.2.1)} or \cite{AHK07}*{Lemma 1.5}.
Finally, we define
\begin{equation} \label{eq:s(X,Delta)}
    s(X,\Delta):=\min_{(Y,D)}\{ \rho(Y,D)+d(Y,D)\}
\end{equation}
where the minimum runs over all terminal pairs $(Y,D)$ with a projective birational morphism $\pi\colon (Y,D)\to (X,\Delta)$ such that $K_Y+D = \pi^* (K_X+\Delta) +E$ for some $\pi$-exceptional $\mathbb R$-divisor $E\ge 0$.
\end{defn}

Since the last invariant $s(X,\Delta)$ will play a major role in controlling the number of steps in the MMP, we give an explicit upper bound of $s(X,\Delta)$ at least for log smooth pairs. We also bound the invariant $e_+(X,\Delta)$ defined after \eqref{eq:e(X,Delta)}. 

\begin{lem} \label{lem:s(X,Delta) log smooth}
Let $(X,\Delta=\sum b_i B_i)$ be a  log smooth pair of dimension $n$. Let $N$ be a positive integer. Assume that $b_i\le 1-\frac{1}{N}$ for all $i$, $(X,\Delta)$ has at most $N$ strata, each stratum has Picard number at most $N$, and $h_{2n-4}^{\alg}(X)\le N$. Then
\[
e_+(X,\Delta)< N^{n+1}\quad \mathrm{and} \quad  s(X,\Delta) < 3N^{2n+3}.
\]
\end{lem}

\begin{proof}
The result is clear when $N=1$: in this case $\Delta=0$ and $\rho(X)=h_{2n-4}^{\alg}(X)=1$ by assumption, which gives $e_+(X,\Delta)=d(X,\Delta)=0$ (as $(X,\Delta)$ is already terminal) and $s(X,\Delta) = 2$ by definition. Thus we may assume that $N\ge 2$. 

A $\bQ$-factorial terminalization $(Y,D)$ of $(X,\Delta)$ can be constructed by repeatedly blowing up the center of some prime exceptional divisor $E$ over $X$ with $a(E,X,\Delta)\le 0$ until all such divisors have been extracted. Note that the center of such a divisor $E$ is necessarily a stratum at each step (this follows from e.g. the calculations in \cite[Corollary 2.31]{KM98}). For each stratum $W\subseteq B_J$, every divisor $E$ with center $W$ and $a(E,X,\Delta)\le 0$ can also be obtained by a weighted blowup with some weights $m_j\in \bN_+$ along $B_j$. Since 
\[
a(E,X,\Delta)=-1+\sum_{j\in J} (1-b_j)m_j
\]
for the exceptional divisor $E$ of such a weighted blowup, and $b_i\le  1-\frac{1}{N}$ by assumption, we see that $m_j\le N$ and hence for each stratum $W$ there are no more than $N^n$ divisors $E$ with center $W$ and $a(E,X,\Delta)\le 0$. As there at no more than $N$ strata and one of them is $X$, 
we deduce that the $\bR$-divisor $D$ on the terminalization has less than $N^{n+1}$ components, and we only need to blowup less than $N^{n+1}$ times to obtain the terminalization $Y$. In particular, $e_+(X,\Delta)< N^{n+1}$.

At the initial stage (i.e. on $X$), the intersection of any two strata has at most $N$ connected components, and this number does not increase under the blowup process. This implies that after each stratum blowup, the Picard number of a stratum can increase at most $N$, hence after $\ell$ blowups, the Picard number of any stratum is at most $(\ell+1)N$. It follows that every stratum of $(Y,D)$ has Picard number at most $N^{n+2}$; in particular, $\rho(Y)\le N^{n+2}$, and if $\widetilde{B}_Y$ is the normalization of $B_Y:=\Supp(D)$, then
\[
\rho(\widetilde{B}_Y)\le N^{n+2} (N^{n+1}-1)
\]
as $D$ has less than $N^{n+1}$ components. On the other hand, by \cite[Theorem 7.31]{Voi-Hodge-book-I}, if we blowup a stratum of Picard number $\rho$, then $h_{2n-4}^\alg$ increases by at most $\rho+1$. By assumption, $h_{2n-4}^\alg(X)\le N$, hence
\[
h_{2n-4}^\alg(Y)\le h_{2n-4}^\alg(X) + (N+1) + (2N+1) + \dots + ((N^{n+1}-1)N+1) < N^{2n+3}.
\]
Putting the above together we see that
\[
\rho(Y,D) = \rho(Y) + \rho(\widetilde{B}_Y) + h_{2n-4}^\alg(Y) < N^{n+2} + N^{n+2} (N^{N+1}-1) + N^{2n+3} = 2N^{2n+3}.
\]

Next observe that by \cite[Corollary 3.2(iii)]{Kol-flops}, every divisor $E$ over $Y$ with $a(E,Y,D)=a(E,X,\Delta)<1$ that is not an echo of $(Y,D)$ is necessarily obtained by a weighted blowup along some stratum of $(X,\Delta)$, as the assumptions of \textit{loc. cit.} is satisfied after blowing up some strata. For each fixed stratum $W\subseteq B_J$, if $m_j$ ($j\in J$) are the weights along $B_j$, then 
\[
\sum_{j\in J} (1-b_j)m_j = 1+a(E,Y,D)<2.
\]
As $b_i\le 1-\frac{1}{N}$, we have $m_j<2N$, hence there are at most $(2N)^n$ divisors $E$ with $a(E,X,\Delta)<1$ that are centered at any given stratum. Since there are no more than $N$ strata, we see that $d(Y,D)\le 2^n N^{n+1}$ and (since $N\ge 2$)
\[
s(X,\Delta)\le \rho(Y,D)+d(Y,D)< 2N^{2n+3}+2^n N^{n+1}< 3N^{2n+3},
\]
which concludes the proof.
\end{proof}

\subsection{Terminal case} \label{ss:explicit terminal}
Next we prove an explicit bound on the number of steps in a special type of MMP on terminal pairs, which includes all terminal flips in dimension $3$ or $4$. 

We start with some general setup. For a  terminal pair $(X,\Delta)$ of dimension $n$, write $\Delta=\sum_{i=1}^k b_i B_i$ where $0<b_k<\dots<b_1<1$ and the $B_i$'s are reduced divisors without common components, and let $\widetilde{B}_i$ be the normalization of $B_i$. Set $b_0:=1$, $b_{k+1}:=0$, and let 
\begin{equation} \label{eq:coeff set S}
    S:=[0,1]\cap \sum_{i=0}^{k} b_i\mathbb N\subseteq \bR.
\end{equation}
In particular, $S=\{0,1\}$ if $\Delta=0$. We define a sequence of topological invariants as follows.
\begin{equation}
    \rho_i:=\rho_i(X,\Delta):=\begin{cases}
        \rho(X), & \mbox{if } i=0, \\
        \rho(\widetilde{B}_i), & \mbox{if } 1\le i\le k, \mbox{ and}\\
        h^{\mathrm{alg}}_{2n-4}(X), & \mbox{if } i=k+1.
    \end{cases}
\end{equation}
By definition, $\rho(X,\Delta)=\sum_{i=0}^{k+1}\rho_i (X,\Delta)$ (\textit{cf.} \eqref{eq:rho(X,Delta)}).
For each $i=1,\dots,k+1$, we also define the $i$-th (weighted) difficulty invariant by
\begin{equation}
    d_i:=d_i(X,\Delta):=\sum_{\eta\in S,\,\eta\ge b_i} \#\left\{ 
    \begin{tabular}{c}
         \mbox{Exceptional divisors $E$ over $X$ with}  \\
         \mbox{$a(E,X,\Delta)<1-\eta$ whose center $c_X(E)$ is not a} \\
         \mbox{codimension one point in the smooth locus of $\Supp(\Delta)$}
    \end{tabular}
    \right\}.
\end{equation}
Clearly $d_1\le d_2\le \cdots\le d_{k+1}\le |S|\cdot d(X,\Delta)$, where $d(X,\Delta)$ is defined in \eqref{eq:d(X,Delta)}. The importance of these invariants comes from the following result.

\begin{thm} \label{thm:invariant seq lexicographic}
Let $(X,\Delta)$ be a $\bQ$-factorial terminal pair of dimension $n\in\{3,4\}$. Then in an $(K_X+\Delta)$-MMP that does not contract any component of $\Delta$, the sequence $(\rho_0,d_1,\rho_1,d_2,\dots,d_{k+1},\rho_{k+1})$ is strictly decreasing in lexicographic order.
\end{thm}

\begin{proof}
This follows from the proof \cites{Fuj04,Fuj05} that such MMP terminates.
\end{proof}

This suggests that if we define an ``$M$-adic difficulty'' by
\begin{align} \label{eq:s_M}
    \delta_M(X,\Delta) & :=\sum_{i=0}^{k+1} M^{2k+2-2i} \rho_i(X,\Delta) + \sum_{i=1}^{k+1} M^{2k+3-2i} d_i(X,\Delta) \\
    & = M^{2k+2} \rho_0 + M^{2k+1} d_1 + M^{2k} \rho_1 +\dots + M d_{k+1}+\rho_{k+1}, \nonumber
\end{align}
where $M$ is a positive integer, then it is reasonable to expect that $\delta_M(X,\Delta)$ is also strictly decreasing in the MMP in Theorem 6.7 when $M\gg 0$. The main result of this subsection is an explicit form of such a statement.

\begin{thm} \label{thm:terminal s_M non-increasing}
Let $(X,\Delta=\sum_{i=1}^k b_i B_i)$ be a $\bQ$-factorial terminal pair of dimension $n$ with $b_k<\dots<b_1$. Set $b_0=1$, $b_{k+1}=0$ and $S=[0,1]\cap \sum_{i=0}^k b_i\mathbb N$ as above. Then for any
\[
M\ge \left( 2+ \left\lceil \frac{1}{b_k}\right\rceil \cdot \left\lceil \frac{1}{1-b_1} \right\rceil \right) \cdot |S|,
\]
the invariant $\delta_M(X,\Delta)$ is non-increasing in any $(K_X+\Delta)$-MMP that does not contract any component of $\Delta$. Moreover, it strictly decreases in the MMP if $n\in\{3,4\}$.
\end{thm}

\begin{cor} \label{cor:terminal explicit}
Let $(X,\Delta=\sum_{i=1}^k b_i B_i)$ be a $\bQ$-factorial terminal pair of dimension $n\in\{3,4\}$ with $b_k<\dots<b_1$. Let $b_0=1$, $b_{k+1}=0$ and $S=[0,1]\cap \sum_{i=0}^k b_i \bN$. Then every $(K_X+\Delta)$-MMP that does not contract any component of $\Delta$ terminates after at most
\[
\left( 2+ \left\lceil \frac{1}{b_k}\right\rceil \cdot \left\lceil \frac{1}{1-b_1} \right\rceil \right)^{2k+2}  |S|^{2k+2} (\rho(X,\Delta)+d(X,\Delta))
\]
steps, where $\rho(X,\Delta)$ and $d(X,\Delta)$ are defined in Definition \ref{defn:rho,d,s}.
\end{cor}

The remaining part of this subsection is devoted to proving the two results above.

\begin{proof}[Proof of Theorem \ref{thm:terminal s_M non-increasing}]

By induction, it suffices to prove that $\delta_M(X,\Delta)\ge \delta_M(X',\Delta')$ for any step $\varphi\colon (X,\Delta)\dashrightarrow (X',\Delta')$ of the $(K_X+\Delta)$-MMP which does not contract any component of $\Delta$, and that the inequality is strict if $n\in\{3,4\}$. Let $\rho'_i:=\rho_i(X',\Delta')$ and $d'_i:=d_i(X',\Delta')$.

First assume that $\varphi$ is a flip. Let $f\colon X\to Z$ be the flipping contraction and $f'\colon X'\to Z$ the flip. Let $B'_i$ be the strict transform of $B_i$ on $X'$ and $\widetilde{B}'_i$ its normalization, and let $\bar{B}_i$ be the normalization of $f(B_i)=f'(B'_i)\subseteq Z$. For ease of notation we also denote the induced morphism $\widetilde{B}_i\to \bar{B}_i$ and $\widetilde{B}'_i\to \bar{B}_i$ by $f$ and $f'$.
\[
\begin{tikzcd}[column sep=small, row sep=small, outer sep=-1.5pt]
\widetilde{B}_i \arrow[r] & B_i \arrow[r,phantom,description,"\subseteq"] & X \arrow[rr,dashed,"\varphi"] \arrow[dr,"f"] & & X' \arrow[dl,"f'"']  \arrow[r,phantom,description,"\supseteq"] & B'_i & \widetilde{B}'_i \arrow[l] \\
& \bar{B}_i \arrow[r] & f(B_i) \arrow[r,phantom,description,"\subseteq"] & Z
\end{tikzcd}
\]
We have $\rho_0=\rho'_0$ since $\varphi$ is small. To compare the other invariants $\rho'_i,d'_i$ with $\rho_i,d_i$, we introduce two intermediate sequences as follows.
\begin{equation}
    \bar{\rho}_i := \begin{cases}
    \rho(\bar{B}_i), & \mbox{if } 1\le i\le k, \\
    h^{\mathrm{alg}}_{2n-4}(Z), & \mbox{if } i=k+1.
    \end{cases}
\end{equation}
\begin{equation}
    \bar{d}_i := \sum_{\eta\in S,\eta\ge b_i} \#\left\{ 
    \begin{tabular}{c}
         \mbox{Exceptional divisors $E$ over $X'$ with}  \\
         \mbox{$a(E,X',\Delta')<1-\eta$ whose center $c_X(E)$ on $X$ is not a} \\
         \mbox{codimension one point in the smooth locus of $\Supp(\Delta)$}
    \end{tabular}
    \right\}.
\end{equation}
Note that in the definition of $\bar{d}_i$, the discrepancy is computed with respect to $(X',\Delta')$ while the center is taken on $X$. It is clear that  $\rho_i\ge \bar{\rho}_i$ for all $i=1,\dots,k+1$. Since $a(E,X,\Delta)\le a(E,X',\Delta')$ for all prime divisors $E$ over $X$ by the negativity lemma, we also have $d_i\ge \bar{d}_i$ for all $i=1,\dots,k+1$. We claim that
\begin{equation} \label{eq:difficulty in flip}
    d'_i - \bar{d}_i \le (M-1) \sum_{1\le j<i} (\rho_j - \bar{\rho}_j),\, \mbox{ and }
\end{equation}
\begin{equation} \label{eq:rho in flip}
    \rho'_i - \bar{\rho}_i \le (M-1)\left( d_i - \bar{d}_i + \sum_{1\le j<i} (\rho_j - \bar{\rho}_j) \right).
\end{equation}
Roughly speaking, the claim says that the sequence $(\rho_0,d_1,\rho_1,d_2,\dots,d_{k+1},\rho_{k+1})$ behaves like the digits of $M$-adic numbers: each time a higher digit drops by $1$, the lower digits increase by at most $M-1$. Taking \eqref{eq:difficulty in flip} and \eqref{eq:rho in flip} for granted, Lemma \ref{lem:elementary inequality} then implies that $\delta_M(X,\Delta)\ge \delta_M(X',\Delta')$. Moreover, equality holds if and only if $d'_i=d_i$ for all $i=1,\dots,k+1$ and $\rho'_j=\rho_j$ for all $j=0,1,\dots,k+1$. This is not possible in dimension $3$ and $4$ by Theorem \ref{thm:invariant seq lexicographic}, thus $\delta_M(X,\Delta)> \delta_M(X',\Delta')$ when $n\in \{3,4\}$.

We now focus on the proof of \eqref{eq:difficulty in flip} and \eqref{eq:rho in flip}.

\begin{proof}[Proof of \eqref{eq:difficulty in flip}]
Let $E$ be an exceptional divisor that has more contributions to $d'_i$ than to $\bar{d}_i$. Let $\xi=c_X(E)$ (resp. $\xi'=c_{X'}(E)$) be its center on $X$ (resp. $X'$). Then $\xi'$ is not a codimension one point in the smooth locus of $\Supp(\Delta')$ and $a(E,X',\Delta')<1-b_i$, otherwise $E$ has no contribution to $d'_i$. On the other hand, since $a(E,X,\Delta)\le a(E,X',\Delta')$, we see that $\xi$ must be a codimension one point in the smooth locus of $\Supp(\Delta)$, otherwise it cannot have more contributions to $d'_i$. Thus $\xi$ is contained in the exceptional locus of $f$. We also have $a(E,X,\Delta)\le a(E,X',\Delta')<1-b_i$, hence $\xi \in B_j$ for some $j<i$ by e.g. \cite{AHK07}*{Lemma 1.5}. By Lemma \ref{lem:compare rho}, there are at most $\rho_j - \bar{\rho}_j$ components of dimension $n-2$ in the exceptional locus of $f|_{B_j}$, and by \cite{AHK07}*{Lemma 1.5}, over the generic point of every such component that is contained in the smooth locus of $B_j$ there are at most $\left\lceil \frac{1}{1-b_j} \right\rceil\le\left\lceil \frac{1}{1-b_1} \right\rceil$ divisors $E$ with $a(E,X,\Delta)<1$. It follows that there are at most 
\[
\left\lceil \frac{1}{1-b_1} \right\rceil\cdot \sum_{1\le j<i} (\rho_j - \bar{\rho}_j)
\]
exceptional divisors $E$ that contribute positively to $d'_i - \bar{d}_i$. Note that every such divisor $E$ is counted at most $|S|$ times in $d'_i$. Thus
\[
d'_i - \bar{d}_i \le |S|\cdot \left\lceil \frac{1}{1-b_1} \right\rceil\cdot \sum_{1\le j<i} (\rho_j - \bar{\rho}_j) \le (M-1) \sum_{1\le j<i} (\rho_j - \bar{\rho}_j),
\]
which gives \eqref{eq:difficulty in flip}.
\end{proof}

\begin{proof}[Proof of \eqref{eq:rho in flip}]
For notational convenience, set $\widetilde{B}'_{k+1} = B'_{k+1} = X'$. By Lemmas \ref{lem:compare rho} and \ref{lem: compare h^alg}, we know that $\rho'_i - \bar{\rho}_i$ is at most the number of irreducible components of dimension $n-2$ in the exceptional locus of $f'|_{\widetilde{B}'_i}$. We first count the number of irreducible components of dimension $n-2$ in the exceptional locus of $f'|_{B'_i}$ (that is, before passing to the normalization). Let $W$ be one such component. Note that $(X',\Delta')$ is terminal, hence $X'$ is smooth at the generic point of $W$. Let $E$ be the exceptional divisor obtained by blowing up $W$ near the generic point of $W$ on $X'$. Then by \cite{Fuj04}*{Lemma 2.2}, we have
\begin{equation} \label{eq:a(E)<1-b}
    a(E,X,\Delta)<a(E,X',\Delta')=1-\eta
\end{equation}
for some $b_i\le \eta\in S$.

Let $\xi=c_X(E)$ be the center of $E$ on $X$. Then $\xi\in \Ex(f)$ since $c_{X'}(E)\in W\subseteq  \Ex(f')$. If $\xi$ is not a codimension one point in the smooth locus of $\Supp(\Delta)$, then \eqref{eq:a(E)<1-b} implies that $E$ has more contributions to $d_i$ than to $\bar{d}_i$, and there are at most $d_i-\bar{d}_i$ such divisors. If $\xi$ is a codimension one point in the smooth locus of $\Supp(\Delta)$, then by \eqref{eq:a(E)<1-b} and \cite{AHK07}*{Lemma 1.5} we see that $\xi\in  B_j$ for some $1\le j<i$. Since $\xi\in \Ex(f)$, it is also the generic point of a component $V$ of dimension $n-2$ in the exceptional locus of $f$. Combined with Lemma \ref{lem:compare rho}, we deduce that there are at most $\sum_{1\le j<i} (\rho_j - \bar{\rho}_j)$ such components $V$. Moreover, once $V$ is fixed, there are at most $\left\lceil \frac{1}{1-b_j} \right\rceil\le\left \lceil \frac{1}{1-b_1} \right\rceil$ divisors $E$ centered at the generic point of $V$ such that $a(E,X,\Delta)<1$. Putting these together, and since $W$ is uniquely determined by $E$ as its center, we see that there are at most 
\[
d_i-\bar{d}_i +\left\lceil \frac{1}{1-b_1} \right\rceil \sum_{1\le j<i} (\rho_j - \bar{\rho}_j)
\]
components $W$ of dimension $n-2$ in the exceptional locus of $f'|_{B'_i}$. When $i=k+1$ this already gives an upper bound of $\rho'_i - \bar{\rho}_i$ by Lemma \ref{lem: compare h^alg}. If $1\le i\le k$, then since $(X',\Delta')$ is terminal, we have $\mult_W \Delta' < 1$ and hence $\mult_W B'_i < \frac{1}{b_i}$. It follows that the preimage of $W$ in the normalization $\widetilde{B}'_i$ has at most $\left\lceil \frac{1}{b_i} \right\rceil \le \left\lceil \frac{1}{b_k} \right\rceil$ components. In other words, the contribution to $\rho'_i - \bar{\rho}_i$ by each component $W$ is at most $\left\lceil \frac{1}{b_k} \right\rceil$. Therefore,
\begin{align*}
\rho'_i - \bar{\rho}_i & \le \left\lceil \frac{1}{b_k} \right\rceil\left( d_i - \bar{d}_i + \left\lceil \frac{1}{1-b_1} \right\rceil\sum_{1\le j<i} (\rho_j - \bar{\rho}_j) \right) \\
& \le (M-1)\left( d_i - \bar{d}_i + \sum_{1\le j<i} (\rho_j - \bar{\rho}_j) \right)
\end{align*}
as desired.
\end{proof}

As discussed before, this concludes the proof of the theorem when $\varphi$ is a flip.

Next assume that $\varphi\colon (X,\Delta)\to (X',\Delta')$ is a divisorial contraction. Then clearly $\rho'_0=\rho_0-1$ and $\rho'_i\le \rho_i$ for all $i\ge 1$. Since none of the components of $\Delta$ is contracted, the pair $(X',\Delta')$ remains terminal. Set $\bar{\rho}_i := \rho'_i$ and $\bar{d}_i:=d_i$. Similar to the previous case, we claim that
\begin{equation}
    d'_i - \bar{d}_i \le (M-1)\sum_{0\le j<i} (\rho_j - \bar{\rho}_j),
\end{equation}
which implies, again by Theorem \ref{thm:invariant seq lexicographic} and Lemma \ref{lem:elementary inequality}, that $\delta_M(X,\Delta)\ge \delta_M(X',\Delta')$, with strict inequality when $n\in \{3,4\}$.

To see this claim, let $F$ be the (unique) exceptional divisor of $\varphi$ and let $E$ be an exceptional divisor over $X'$ that contributes more to $d'_i$ than to $\bar{d}_i  = d_i$. Let $\xi=c_X(E)$ (resp. $\xi'=c_{X'}(E)$) be the center of $E$ on $X$ (resp. $X'$). Then $\xi'$ is not a codimension one point of the smooth locus of $\Supp(\Delta')$, and $a(E,X',\Delta')<1-b_i$. Since $a(E,X,\Delta)\le a(E,X',\Delta')$, in order for the divisor $E$ to contribute less to $d_i$, there are only two possibilities: either $E=F$ (so it is not exceptional over $X$), or $E$ is exceptional over $X$ but its center $\xi$ is a codimension one point of the smooth locus of $\Supp(\Delta)$. In both cases we necessarily have $\xi\in \Ex(\varphi)= F$. Let us analyze the latter situation more carefully.

Since $a(E,X,\Delta)\le a(E,X',\Delta')<1-b_i$, by \cite{AHK07}*{Lemma 1.5} we know that $\xi\in B_j$ for some $1\le j<i$. Note that $\xi'=\varphi(\xi)$. If $\xi'$ is a point of codimension at least $3$ in $X'$, then $W=\overline{\{\xi\}}$ is a component of dimension $n-2$ in the exceptional locus of $\varphi|_{B_j}$. There are at most $\rho_j - \rho'_j = \rho_j - \bar{\rho}_j$ such components by Lemma \ref{lem:compare rho}. Once $W$ is fixed, by \cite{AHK07}*{Lemma 1.5}, there are at most $\left\lceil \frac{1}{1-b_j} \right\rceil\le \left\lceil \frac{1}{1-b_1} \right\rceil$ exceptional divisors $E$ with center $W$ and $a(E,X,\Delta)<1$. As $j$ varies these give at most
\[
\left\lceil \frac{1}{1-b_1}\right\rceil \sum_{1\le j<i} (\rho_j - \bar{\rho}_j)
\]
candidates of the divisor $E$ in total. If $\xi'$ has codimension $2$ in $X'$, then $\xi'$ is the generic point of $\varphi(F)$ (because $\xi\in F$ and $\varphi(F)$ has codimension at least $2$ in $X$) and $\varphi|_{B_j}$ is finite around $\xi$. As $(X',\Delta')$ is terminal, we get $\mult_{\xi'} \Delta' <1$ and thus $\mult_{\xi'} (B'_1+\dots+B'_j)< \frac{1}{b_j}\le \frac{1}{b_k}$. It follows that the preimage of $\xi'$ in the normalization $\widetilde{B}_1\cup \dots\cup \widetilde{B}_j$ has at most $\left\lceil \frac{1}{b_k} \right\rceil$ components, giving rise to at most $\left\lceil \frac{1}{b_k} \right\rceil$ possibilities of $\xi$. For each such choice of $\xi$, there are at most $\left\lceil \frac{1}{1-b_j} \right\rceil\le\left\lceil \frac{1}{1-b_1} \right\rceil$ divisors $E$ with center $\xi$ such that $a(E,X,\Delta)<1$ (by \cite{AHK07}*{Lemma 1.5} as before). Putting these together, we conclude that there are at most
\[
\left\lceil \frac{1}{1-b_1} \right\rceil \cdot \left\lceil \frac{1}{b_k} \right\rceil + \left\lceil \frac{1}{1-b_1} \right\rceil \sum_{1\le j<i} (\rho_j - \bar{\rho}_j) = \left\lceil \frac{1}{1-b_1} \right\rceil \cdot \left( \left\lceil \frac{1}{b_k} \right\rceil + \sum_{1\le j<i} (\rho_j - \bar{\rho}_j) \right)
\]
exceptional divisors over $X$ that contribute more to $d'_i$ than to $d_i$. Recall that $\rho_0 - \rho'_0 = 1$ and there is also the option that $E=F$. As every divisor $E$ is counted at most $|S|$ times in $d'_i$, we finally deduce that
\[
d'_i - \bar{d}_i \le |S| \cdot \left( 1 + \left\lceil \frac{1}{1-b_1} \right\rceil \cdot \left( \left\lceil \frac{1}{b_k} \right\rceil + \sum_{1\le j<i} (\rho_j - \bar{\rho}_j) \right) \right) \le (M-1) \sum_{0\le j<i} (\rho_j - \bar{\rho}_j).
\]
This proves the claim and hence of the theorem as well.
\end{proof}

The following elementary lemma has been used in the proof above.

\begin{lem} \label{lem:elementary inequality}
Let $a_i,b_i,c_i\in \bR\,\, (i=0,\dots,m)$ and $M>1$ be such that $a_i\ge b_i$ and $c_i-b_i\le (M-1)\sum_{0\le j<i} (a_j-b_j)$ for all $i$. Then
\[
\sum_{i=0}^m M^{-i} c_i \le \sum_{i=0}^m M^{-i} a_i
\]
with equality if and only if $a_i=b_i=c_i$ for all $i$.
\end{lem}

In the applications above, the sequence $a_i$ (resp. $b_i$, resp. $c_i$) represents the sequence $(\rho_0,d_1,\dots,d_{k+1},\rho_{k+1})$ (resp. $(\bar{\rho}_0,\bar{d}_1,\dots,\bar{d}_{k+1},\bar{\rho}_{k+1})$, resp. $(\rho'_0,d'_1,\dots,d'_{k+1},\rho'_{k+1})$).

\begin{proof}
We may assume that $b_i=0$ by replacing $a_i$ (resp. $c_i$) with $a_i-b_i$ (resp. $c_i-b_i$). The assumption then becomes $a_i\ge 0$ and $c_i\le (M-1)\sum_{0\le j<i} a_j$. It follows that $M^{-i}c_i \le (1-\frac{1}{M}) \sum_{0\le j<i} M^{1-i}a_j$, hence
\begin{align*}
  \sum_{i=0}^m M^{-i}c_i &\le \left(1-\frac{1}{M}\right) \sum_{i=0}^m \sum_{0\le j<i} M^{1-i}a_j \le \left(1-\frac{1}{M}\right) \sum_{j=0}^{m-1} \sum_{i=j}^{+\infty} M^{-i} a_j\\
  &= \sum_{j=0}^{m-1} M^{-j} a_j\le \sum_{j=0}^{m} M^{-j} a_j.  
\end{align*}
The second and the last equality hold if and only if $a_i=0$ for all $i$. If this is the case, then the first equality holds if and only if $c_i=(M-1)\sum_{0\le j<i} a_j = 0$ for all $i$.
\end{proof}

To prove Corollary \ref{cor:terminal explicit}, we need another elementary statement.

\begin{lem} \label{lem:s_M vs rho+d}
We have $\rho(X,\Delta)+d(X,\Delta)\le \delta_M(X,\Delta)$. Assume that $M\ge |S|+1$. Then we also have 
\begin{equation} 
    \delta_M(X,\Delta)\le M^{2k+2} (\rho(X,\Delta) + d(X,\Delta)).
\end{equation}
\end{lem}

\begin{proof}
The first inequality is clear from the definition \eqref{eq:s_M} of $\delta_M(X,\Delta)$. If $M\ge |S|+1$, then as $d_i\le |S|\cdot d(X.\Delta)$ for all $i$, from \eqref{eq:s_M} we also have
\begin{align*}
    \delta_M(X,\Delta) & \le M^{2k+2}\rho(X,\Delta) + (M^{2k+1}+M^{2k-1}+\dots+M)\cdot |S|\cdot d(X,\Delta) \\
    & \le M^{2k+2}\rho(X,\Delta) + (M^{2k+1}+M^{2k-1}+\dots+M)\cdot (M-1)\cdot d(X,\Delta) \\
    & \le M^{2k+2} (\rho(X,\Delta) + d(X,\Delta)),
\end{align*}
which gives the second inequality.
\end{proof}

\begin{proof}[Proof of Corollary \ref{cor:terminal explicit}]
Since $\delta_M(X,\Delta)\in \bN$ by definition, we deduce from Theorem \ref{thm:terminal s_M non-increasing} that the MMP terminates in at most $\delta_M(X,\Delta)$ steps, where
\[
M := \left( 2+ \left\lceil \frac{1}{b_k}\right\rceil \cdot \left\lceil \frac{1}{1-b_1} \right\rceil \right) \cdot |S|.
\]
The result then follows from Lemma \ref{lem:s_M vs rho+d}.
\end{proof}

\subsection{MMP with discrete discrepancies}

Our next goal is to give explicit bounds for klt MMP in dimension $3$, or in dimension $4$ when one of $K_X+\Delta$ or $\Delta$ is big. A common feature in these cases is that the log discrepancies belong to some fixed finite set during the MMP (\textit{cf.} \cite[Section 3]{HQZ25}). To give a more uniform treatment, we first prove an explicit termination statement by fixing this set. See Definition \ref{defn:rho,d,s} for the definition of $s(X,\Delta)$ and $e(X,\Delta)$ that appear in the statement below.


\begin{lem} \label{lem:fix discrep explicit}
Let $S\subseteq [0,1]$ be a finite set such that $0,1\in S$ and for any $a,b\in S$ with $a+b\le 1$, we have $a+b\in S$. 
Let $(X,\Delta)$ be a klt pair of dimension $n\in \{3,4\}$. Let
\begin{equation} \label{eq:MMP discrete discrep}
\begin{tikzcd}
(X,\Delta)=:(X_0,\Delta_0) \arrow[r,dashed,"\varphi_0"] & (X_1,\Delta_1) \arrow[r,dashed,"\varphi_1"] & \cdots \arrow[r,dashed] & (X_i,\Delta_i) \arrow[r,dashed,"\varphi_i"] & \cdots 
\end{tikzcd}
\end{equation}
be an $(K_X+\Delta)$-MMP such that for any divisor $E$ over $X$ and any $i$ with $a(E,X_i,\Delta_i)<0$, we have $-a(E,X_i,\Delta_i)\in S$. 
Then this MMP terminates after at most
\[
M^{2(k+1)^2 |S|} s(X,\Delta)
\]
steps, where $k=e_+(X,\Delta)$ and $M=\left(2+\left\lceil \frac{1}{\min S\setminus\{0\}}\right\rceil\cdot \left\lceil \frac{1}{1-\max S\setminus\{1\}} \right\rceil\right)\cdot |S|$.
\end{lem}

Before proving this we make an elementary observation.

\begin{lem} \label{lem:drop coeff}
Let $(X,\Delta)$ be a terminal pair and let $0\le D\le \Delta$ be an $\bR$-divisor such that $K_X+D$ is $\bR$-Cartier. Then $\rho(X,D)\le \rho(X,\Delta)$ and $d(X,D)\le d(X,\Delta)$.
\end{lem}

\begin{proof}
As $0\le D\le \Delta$, every component of $D$ is also a component of $\Delta$, hence the first inequality is clear from the definition. Since $a(E,X,\Delta)\le a(E,X,D)$ holds for any prime divisor $E$ over $X$, to prove the second inequality it suffices to show that if $E$ is an exceptional divisor with $a(E,X,D)<1$ and it is not an echo of $(X,D)$, then it is also not an echo of $(X,\Delta)$. Suppose that $E$ is an echo for $(X,\Delta)$, then its center $\xi$ is a codimension one smooth point of $\Supp(\Delta)$. Note that in this case $\xi\in \Supp(D)$, otherwise $a(E,X,D)=a(E,X)\ge 1$ since $X$ is terminal and hence smooth at $\xi$. It follows that $\xi$ is also a codimension one smooth point of $\Supp(D)$ and hence an echo for $(X,D)$, a contradiction. This gives the second inequality.
\end{proof}

\begin{proof}[Proof of Lemma \ref{lem:fix discrep explicit}]
Let $\pi\colon (Y,D)\to (X,\Delta)$ be a projective birational morphism from a $\bQ$-factorial terminal pair such that $K_Y+D = \pi^*(K_X+\Delta)+E$ for some $\pi$-exceptional $\bR$-divisor $E\ge 0$. In other words, $\mult_F(D)\ge \min\{0,-a(F,X,\Delta)\}$ for all prime divisors $F$ on $Y$. We need to prove that the MMP \eqref{eq:MMP discrete discrep} terminates after at most $M^{2(k+1)^2 |S|}(\rho(Y,D)+d(Y,D))$ steps. By Lemma \ref{lem:drop coeff}, we may decrease the coefficients of $D$ and assume that
\[
\mult_F(D)=\min\{0,-a(F,X,\Delta)\}
\]
for all prime divisors $F$ on $Y$. In particular, we may assume that $\Coef(D)\subseteq S$. 

By Lemma \ref{lem:lift MMP terminal}, we can lift the given $(K_X+\Delta)$-MMP to a diagram 
\[
\begin{tikzcd}
(Y,D) =: (Y_0,D_0) \arrow[r,dashed,"\psi_0"] \arrow[d, shift left=5ex, "\pi_0=\pi"'] & (Y_1,D_1) \arrow[r,dashed,"\psi_1"] \arrow[d,"\pi_1"'] & \cdots \arrow[r,dashed] & (Y_i,D_i) \arrow[r,dashed,"\psi_i"] \arrow[d,"\pi_i"'] & \cdots \\
(X,\Delta) =: (X_0,\Delta_0) \arrow[r,dashed,"\varphi_0"] & (X_1,\Delta_1) \arrow[r,dashed,"\varphi_1"] & \cdots \arrow[r,dashed] & (X_i,\Delta_i) \arrow[r,dashed,"\varphi_i"] & \cdots 
\end{tikzcd}
\]
where each $\pi_i$ ($i\ge 1$) is a $\bQ$-factorial terminalization of $(X_i,\Delta_i)$ and each $\psi_i$ ($i\ge 0$) is composed of a sequence of $(K_{Y_i}+B_i)$-MMP where $B_i=(\psi_i^{-1})_* D_{i+1}\le D_i$. In particular, $\psi_i$ does not contract any component of $B_i$. Let $\delta_i:=\delta_M(Y_i,D_i)$ which is defined by \eqref{eq:s_M}. At each step we have two possibilities:
\begin{enumerate}
    \item $B_i=D_i$. In this case $\delta_{i+1}\le \delta_i-1$ by Theorem \ref{thm:terminal s_M non-increasing}. 
    \item $B_i\neq D_i$. By Lemma \ref{lem:drop coeff} and Lemma \ref{lem:s_M vs rho+d}, we have $\delta_M(Y_i,B_i)\le M^{2k+2} \delta_M(Y_i,D_i)$, hence Theorem \ref{thm:terminal s_M non-increasing} gives 
    $\delta_{i+1}\le \delta_M(Y_i,B_i)\le  M^{2k+2}\delta_i$ (note that in this case it might be possible that $Y_{i+1}=Y_i$).
\end{enumerate}
By construction $\Supp(D_i)$ has only $k$ components, and $\Coef(B_i)\cup \Coef(D_i)\subseteq S$ for all $i$, hence the second case can only happen at most $k\cdot |S|$ times since we always have $0\le B_i\le D_i$. If we rescale the sequence $\delta_0,\delta_1,\dots$ by setting $\tilde \delta_i := M^{(2k+2)\ell_i}\delta_i$ where $\ell_i = \#\{j\ge i\,|\,B_j\neq D_j\}$, then the above analysis gives $\tilde \delta_{i+1} \le \tilde \delta_i$ for all $i$ and the  equality holds at most $k\cdot |S|$ times. From here we see that the MMP \eqref{eq:MMP discrete discrep} must terminate after at most 
\[
\tilde \delta_0 \le M^{(2k+2)k\cdot |S|} \delta_0 + k\cdot |S|\le M^{2(k+1)^2 |S|} (\rho(Y,D)+d(Y,D))
\]
steps, where the last inequality is by Lemma \ref{lem:s_M vs rho+d}. This proves the desired result.
\end{proof}

We now put together the results obtained in this section and give some concrete applications. We shall only consider MMP that starts with a log smooth pair, since for any klt pair $(X,\Delta)$, we can always choose a log resolution $\pi\colon Y\to X$ and a log smooth klt pair $(Y,D)$ such that $K_Y+D\ge \pi^*(K_X+\Delta)+E$ for some $\pi$-exceptional $\bR$-divisor $E\ge 0$, and any $(K_X+\Delta)$-MMP can be lifted to a $(K_Y+D)$-MMP by Lemma \ref{lem:lift MMP dlt}. Moreover, if $K_X+\Delta$ (resp. $\Delta$) is big, then we can choose the log smooth model $(Y,D)$ so that $K_Y+D$ (resp. $D$) is big.

The first application is about fourfold general type MMPs.

\begin{thm}\label{thm:4fold explicit general type}
Let $N$ be a positive integer and let $(X,\Delta)$ be a projective log smooth klt pair of dimension $4$ such that $N\Delta$ has integer coefficients. Assume that 
\begin{enumerate}
    \item $(X,\Delta)$ has at most $N$ strata, each stratum has Picard number at most $N$, and $h_4^{\alg}(X)\le N$,
    \item $K_X+\Delta$ is big, $\vol(K_X+\Delta)\ge \frac{1}{N}$, and $((K_X+\Delta)\cdot H^3)\le N$ for some very ample divisor $H$ on $X$.
\end{enumerate}
Then any $(K_X+\Delta)$-MMP terminates after at most $M^M$ steps, where $M:=(2 N)^9!$.
\end{thm}
\begin{proof}
By \cite[Corollary 3.20]{HQZ25}, we know that in any $(K_X+\Delta)$-MMP, the Cartier index of any $\bQ$-Cartier Weil divisor is at most $(2^8 N^9)!$. Since $N\Delta$ has integer coefficients, we see that the assumptions of Lemma \ref{lem:fix discrep explicit} are satisfied if we set
\[
S:=[0,1]\cap \frac{1}{N\cdot (2^8 N^9)!}\bR.
\]
Let $N_1:=N\cdot (2^8 N^9)!$. Then $|S|=N_1+1$ and $\min (S\setminus\{0\})=1-\max (S\setminus\{1\})=\frac{1}{N_1}$. By Lemma \ref{lem:s(X,Delta) log smooth}, we also have $e(X,\Delta)\le N^5-1$ and $s(X,\Delta)\le 3N^{11}$. Thus by Lemma \ref{lem:fix discrep explicit} we see that the MMP terminates after at most 
\[
(2N_1+N_1^3)^{2N^{10} (N_1+1)} \cdot 3N^{11}\le  M^M
\]
steps.
\end{proof}

The same argument also works when the boundary $\Delta$ is big.

\begin{thm}\label{thm:4fold explicit boundary big}
Let $N$ be a positive integer and let $(X,\Delta)$ be a projective log smooth klt pair of dimension $4$ such that $N\Delta$ has integer coefficients. Assume that 
\begin{enumerate}
    \item $(X,\Delta)$ has at most $N$ strata, each stratum has Picard number at most $N$, and $h_4^{\alg}(X)\le N$,
    \item $\Delta$ is big, $\vol(\Delta)\ge \frac{1}{N}$, and $(\Delta\cdot H^3)\le N$ for some very ample divisor $H$ on $X$.
\end{enumerate}
Then any $(K_X+\Delta)$-MMP terminates after at most $M^M$ steps, where $M:=(2 N)^9!$.
\end{thm}
\begin{proof}
This follows from the same proof of Theorem \ref{thm:4fold explicit general type}, replacing \cite[Corollary 3.20]{HQZ25} by \cite[Corollary 3.18]{HQZ25}.
\end{proof}


Finally, we give an explicit bound on the length of threefold MMPs. It is not clear if such an explicit termination should be expected in dimension $\ge 4$, see \cite[Example 12]{Sho-mld-conj} for some discussions.

\begin{thm}\label{thm: threefold explcit upper bound}
Let $N$ be a positive integer and let $(X,\Delta)$ be a projective log smooth klt pair of dimension $3$ such that $N\Delta$ has integer coefficients, $(X,\Delta)$ has at most $N$ strata, and each stratum has Picard number at most $N$. Then any $(K_X+\Delta)$-MMP terminates after at most 
\[
3(N\cdot M!)^{8N^8(M!+1)}
\]
steps, where $M:=((N^5+2N)!)^{2^{N^4}}$.
\end{thm}

\begin{proof}
Since $N\Delta$ has integer coefficients, any positive discrepancy of $(X,\Delta)$ is at least $\frac{1}{N}$. For smooth threefolds we also have $h^{\alg}_2(X)=\rho(X)$ (indeed, the invariant $h^{\alg}_2(X)$ does not play any role in proving the termination of $3$-fold MMP, but we don't need this fact). By Lemma \ref{lem:s(X,Delta) log smooth}, we see that $e(X,\Delta)\le N^4-1$ and $s(X,\Delta)\le 3N^9$. Thus by Corollary \ref{cor:index 3-fold MMP}, the Cartier index of any $\bQ$-Cartier Weil divisor is at most $M$. Then the assumptions of Lemma \ref{lem:fix discrep explicit} are satisfied by $S:=[0,1]\cap \frac{1}{N\cdot M!}\bR$, and it follows that the MMP terminates after at most 
\[
((2+(N\cdot M!)^2)\cdot N\cdot M!)^{2N^8(M!+1)}\cdot 3N^9 \le 3(N\cdot M!)^{8N^8(M!+1)}
\]
steps.
\end{proof}

\bibliography{ref}

\begin{bibdiv}
\begin{biblist}

\bib{AHK07}{article}{
      author={Alexeev, Valery},
      author={Hacon, Christopher},
      author={Kawamata, Yujiro},
       title={Termination of (many) 4-dimensional log flips},
        date={2007},
        ISSN={0020-9910,1432-1297},
     journal={Invent. Math.},
      volume={168},
      number={2},
       pages={433\ndash 448},
         url={https://doi.org/10.1007/s00222-007-0038-1},
}

\bib{BCHM}{article}{
      author={Birkar, Caucher},
      author={Cascini, Paolo},
      author={Hacon, Christopher~D.},
      author={McKernan, James},
       title={Existence of minimal models for varieties of log general type},
        date={2010},
     journal={J. Amer. Math. Soc.},
      volume={23},
      number={2},
       pages={405\ndash 468},
}

\bib{Bir07}{article}{
      author={Birkar, Caucher},
       title={Ascending chain condition for log canonical thresholds and termination of log flips},
        date={2007},
        ISSN={0012-7094,1547-7398},
     journal={Duke Math. J.},
      volume={136},
      number={1},
       pages={173\ndash 180},
         url={https://doi.org/10.1215/S0012-7094-07-13615-9},
      review={\MR{2271298}},
}

\bib{Bir10}{article}{
      author={Birkar, Caucher},
       title={On termination of log flips in dimension four},
        date={2010},
        ISSN={0025-5831,1432-1807},
     journal={Math. Ann.},
      volume={346},
      number={2},
       pages={251\ndash 257},
         url={https://doi.org/10.1007/s00208-009-0396-7},
      review={\MR{2563688}},
}

\bib{Bir12lcflip}{article}{
      author={Birkar, Caucher},
       title={Existence of log canonical flips and a special {LMMP}},
        date={2012},
     journal={Publications Math\'ematiques de l'IH\'ES},
      volume={115},
       pages={325\ndash 368},
         url={https://www.numdam.org/articles/10.1007/s10240-012-0039-5/},
}

\bib{CC24}{article}{
      author={Cascini, Paolo},
      author={Chen, Hsin-Ku},
       title={Chern numbers of terminal threefolds},
        date={2024},
     journal={Int. Math. Res. Not. IMRN},
      number={23},
       pages={14464\ndash 14478},
}

\bib{CH11}{article}{
      author={Chen, Jungkai~A.},
      author={Hacon, Christopher~D.},
       title={Factoring 3-fold flips and divisorial contractions to curves},
        date={2011},
     journal={J. Reine Angew. Math.},
      volume={657},
       pages={173\ndash 197},
}

\bib{CT23}{article}{
      author={Chen, Guodu},
      author={Tsakanikas, Nikolaos},
       title={On the termination of flips for log canonical generalized pairs},
        date={2023},
     journal={Acta Math. Sin. (Engl. Ser.)},
      volume={39},
      number={6},
       pages={967\ndash 994},
}

\bib{Fuj04}{article}{
      author={Fujino, O.},
       title={Termination of $4$-fold canonical flips},
        date={2004},
     journal={Publ. Res. Inst. Math. Sci.},
      volume={40},
      number={1},
       pages={231\ndash 237},
}

\bib{Fuj05}{article}{
      author={Fujino, O.},
       title={Addendum to ``{T}ermination of $4$-fold canonical flips''},
        date={2005},
     journal={Publ. Res. Inst. Math. Sci.},
      volume={41},
      number={1},
       pages={252\ndash 257},
}

\bib{Fuj07a}{incollection}{
      author={Fujino, O.},
       title={Special termination and reduction to pl flips},
        date={2007},
   booktitle={Flips for 3--folds and 4--folds},
   publisher={Oxford University Press},
}

\bib{He23}{article}{
      author={He, Yang},
       title={On the strong {S}arkisov program},
        date={2024},
        note={\href{https://arxiv.org/abs/2311.08750}{\textsf{arXiv:2311.08750v2}}},
}

\bib{HL22-Zariski}{article}{
      author={Han, Jingjun},
      author={Li, Zhan},
       title={Weak {Z}ariski decompositions and log terminal models for generalized pairs},
        date={2022},
        ISSN={0025-5874,1432-1823},
     journal={Math. Z.},
      volume={302},
      number={2},
       pages={707\ndash 741},
         url={https://doi.org/10.1007/s00209-022-03073-w},
      review={\MR{4480208}},
}

\bib{HLS24}{article}{
      author={Han, Jingjun},
      author={Liu, Jihao},
      author={Shokurov, Vyacheslav},
       title={{ACC} for minimal log discrepancies of exceptional singularities},
        date={2024},
     journal={Peking Math. J.},
        note={\url{https://link.springer.com/article/10.1007/s42543-024-00091-x}},
}

\bib{HMX-ACC}{article}{
      author={Hacon, Christopher~D.},
      author={McKernan, James},
      author={Xu, Chenyang},
       title={A{CC} for log canonical thresholds},
        date={2014},
     journal={Ann. of Math. (2)},
      volume={180},
      number={2},
       pages={523\ndash 571},
}

\bib{HQZ25}{unpublished}{
      author={Han, Jingjun},
      author={Qi, Lu},
      author={Zhuang, Ziquan},
       title={Boundedness in general type {MMP}},
        date={2025},
        note={{\href{https://arxiv.org/abs/2506.20183v2}{\textsf{arXiv:2506.20183v2}}}},
}

\bib{HX13}{article}{
      author={Hacon, Christopher~D.},
      author={Xu, Chenyang},
       title={Existence of log canonical closures},
        date={2013},
     journal={Invent. Math.},
      volume={192},
      number={1},
       pages={161\ndash 195},
}

\bib{Kaw-termination}{article}{
      author={Kawamata, Yujiro},
       title={Termination of log flips for algebraic {$3$}-folds},
        date={1992},
     journal={Internat. J. Math.},
      volume={3},
      number={5},
       pages={653\ndash 659},
}

\bib{Kim25}{article}{
      author={Kim, Donghyeon},
       title={The number of smooth varieties in an {MMP} on a 3-fold of {F}ano type},
        date={2025},
        note={\href{https://arxiv.org/abs/2502.01370}{\textsf{arXiv:2502.01370}}},
}

\bib{KM98}{book}{
      author={Koll\'{a}r, J\'{a}nos},
      author={Mori, Shigefumi},
       title={Birational geometry of algebraic varieties},
      series={Cambridge Tracts in Mathematics},
   publisher={Cambridge University Press, Cambridge},
        date={1998},
      volume={134},
        note={With the collaboration of C. H. Clemens and A. Corti, Translated from the 1998 Japanese original},
}

\bib{KMM87}{incollection}{
      author={Kawamata, Yujiro},
      author={Matsuda, Katsumi},
      author={Matsuki, Kenji},
       title={Introduction to the minimal model problem},
        date={1987},
   booktitle={Algebraic geometry, {S}endai, 1985},
      series={Adv. Stud. Pure Math.},
      volume={10},
   publisher={North-Holland, Amsterdam},
       pages={283\ndash 360},
}

\bib{Kol13}{book}{
      author={Koll\'{a}r, J\'{a}nos},
       title={Singularities of the minimal model program},
      series={Cambridge Tracts in Mathematics},
   publisher={Cambridge University Press, Cambridge},
        date={2013},
      volume={200},
        note={With a collaboration of S\'{a}ndor Kov\'{a}cs},
}

\bib{Kol21}{article}{
      author={Koll\'ar, J\'anos},
       title={Relative {MMP} without {$\mathbb{Q}$}-factoriality},
        date={2021},
     journal={Electron. Res. Arch.},
      volume={29},
      number={5},
       pages={3193\ndash 3203},
}

\bib{Kol-flops}{article}{
      author={Koll\'ar, J\'anos},
       title={Flops},
        date={1989},
     journal={Nagoya Math. J.},
      volume={113},
       pages={15\ndash 36},
}

\bib{K+-flip-abundance}{book}{
      author={Koll{\'a}r, J{\'a}nos},
      author={others},
       title={Flips and abundance for algebraic threefolds},
   publisher={Soci\'et\'e Math\'ematique de France, Paris},
        date={1992},
        note={Papers from the Second Summer Seminar on Algebraic Geometry held at the University of Utah, Salt Lake City, Utah, August 1991, Ast{\'e}risque No. 211 (1992)},
}

\bib{Mar22}{article}{
      author={Martinelli, Diletta},
       title={Effective bounds for the number of {MMP}-series of a smooth threefold},
        date={2022},
     journal={Glasg. Math. J.},
      volume={64},
      number={1},
       pages={106\ndash 113},
}

\bib{Mor25}{article}{
      author={Moraga, Joaqu\'in},
       title={Termination of pseudo-effective 4-fold flips},
        date={2025},
        ISSN={0025-5874,1432-1823},
     journal={Math. Z.},
      volume={310},
      number={4},
       pages={Paper No. 73, 32},
         url={https://doi.org/10.1007/s00209-025-03766-y},
      review={\MR{4913095}},
}

\bib{Sho-mld-conj}{article}{
      author={Shokurov, Vyacheslav~V.},
       title={Letters of a bi-rationalist. {V}. {M}inimal log discrepancies and termination of log flips},
        date={2004},
     journal={Tr. Mat. Inst. Steklova},
      volume={246},
      number={Algebr. Geom. Metody, Svyazi i Prilozh.},
       pages={328\ndash 351},
}

\bib{Sho09}{article}{
      author={Shokurov, Vyacheslav~V.},
       title={Letters of a bi-rationalist. {VII}. {O}rdered termination},
        date={2009},
     journal={Tr. Mat. Inst. Steklova},
      volume={264},
       pages={184\ndash 208},
}

\bib{Sho85}{article}{
      author={Shokurov, Vyacheslav~V.},
       title={A nonvanishing theorem},
        date={1985},
        ISSN={0373-2436},
     journal={Izv. Akad. Nauk SSSR Ser. Mat.},
      volume={49},
      number={3},
       pages={635\ndash 651},
      review={\MR{794958}},
}

\bib{Sho96}{article}{
      author={Shokurov, Vyacheslav~V.},
       title={3-fold log models},
        date={1996},
     journal={J. Math. Sci.},
      volume={81},
      number={3},
       pages={2667\ndash 2699},
}

\bib{Voi-Hodge-book-I}{book}{
      author={Voisin, Claire},
       title={Hodge theory and complex algebraic geometry. {I}},
      series={Cambridge Studies in Advanced Mathematics},
   publisher={Cambridge University Press, Cambridge},
        date={2007},
      volume={76},
}

\end{biblist}
\end{bibdiv}

\end{document}